\documentclass[reqno, 12pt]{amsart}
\pdfoutput=1
\makeatletter
\let\origsection=\section \def\section{\@ifstar{\origsection*}{\mysection}}
\def\mysection{\@startsection{section}{1}\z@{.7\linespacing\@plus\linespacing}{.5\linespacing}{\normalfont\scshape\centering\S}}
\makeatother

\usepackage{amsmath,amssymb,amsthm}
\usepackage{mathrsfs}
\usepackage{mathabx}\changenotsign
\usepackage{dsfont}
\usepackage{bbm}

\usepackage{xcolor}
\usepackage[backref=section]{hyperref}
\usepackage[ocgcolorlinks]{ocgx2}
\hypersetup{
    colorlinks=true,
    linkcolor={red!60!black},
    citecolor={green!60!black},
    urlcolor={blue!60!black},
}

\usepackage[open,openlevel=2,atend]{bookmark}

\usepackage[abbrev,msc-links,backrefs]{amsrefs}
\usepackage{doi}

\renewcommand{\PrintDOI}[1]{\doi{#1}}

\usepackage[T1]{fontenc}
\usepackage{lmodern}
\usepackage[babel]{microtype}
\usepackage[english]{babel}

\usepackage{geometry}
\numberwithin{equation}{section}
\numberwithin{figure}{section}

\usepackage{enumitem}
\def\rmlabel{\upshape({\itshape \roman*\,})}

\let\polishlcross=\l
\def\l{\ifmmode\ell\else\polishlcross\fi}

\def\tand{\ \text{and}\ }
\def\qand{\quad\text{and}\quad}
\def\qqand{\qquad\text{and}\qquad}

\let\setminus=\smallsetminus

\makeatletter
\def\moverlay{\mathpalette\mov@rlay}
\def\mov@rlay#1#2{\leavevmode\vtop{   \baselineskip\z@skip \lineskiplimit-\maxdimen
   \ialign{\hfil$\m@th#1##$\hfil\cr#2\crcr}}}
\newcommand{\charfusion}[3][\mathord]{
    #1{\ifx#1\mathop\vphantom{#2}\fi
        \mathpalette\mov@rlay{#2\cr#3}
      }
    \ifx#1\mathop\expandafter\displaylimits\fi}
\makeatother

\newcommand{\dcup}{\charfusion[\mathbin]{\cup}{\cdot}}

\DeclareFontFamily{U}  {MnSymbolC}{}
\DeclareSymbolFont{MnSyC}         {U}  {MnSymbolC}{m}{n}
\DeclareFontShape{U}{MnSymbolC}{m}{n}{
    <-6>  MnSymbolC5
   <6-7>  MnSymbolC6
   <7-8>  MnSymbolC7
   <8-9>  MnSymbolC8
   <9-10> MnSymbolC9
  <10-12> MnSymbolC10
  <12->   MnSymbolC12}{}
\DeclareMathSymbol{\powerset}{\mathord}{MnSyC}{180}

\let\epsilon=\varepsilon
\let\eps=\epsilon
\let\rho=\varrho
\let\theta=\vartheta

\let\phi=\varphi

\def\EE{{\mathds E}}

\def\PP{{\mathds P}}

\newcommand{\cA}{\mathcal{A}}
\newcommand{\cB}{\mathcal{B}}
\newcommand{\cC}{\mathcal{C}}

\newcommand{\cK}{\mathcal{K}}

\newcommand{\cQ}{\mathcal{Q}}

\newcommand{\cT}{\mathcal{T}}

\newcommand{\ccG}{\mathscr{G}}

\newtheoremstyle{note}  	{4pt}{4pt}{\sl}{}{\bfseries}{.}{.5em}{}
\newtheoremstyle{introthms}	{3pt}{3pt}{\itshape}{}  {\bfseries}{.}{.5em}{\thmnote{#3}}
\newtheoremstyle{remark}	{2pt}{2pt}{\rm}{}{\bfseries}{.}{.3em}{}

\theoremstyle{plain}
\newtheorem{theorem}{Theorem}[section]
\newtheorem{lemma}[theorem]{Lemma}
\newtheorem{example}[theorem]{Example}

\theoremstyle{note}
\newtheorem{dfn}[theorem]{Definition}
\usepackage{etoolbox}
\AtBeginEnvironment{dfn}{}

\theoremstyle{remark}

\newtheorem{problem}[theorem]{Problem}

\usepackage{lineno}
\newcommand*\patchAmsMathEnvironmentForLineno[1]{
\expandafter\let\csname old#1\expandafter\endcsname\csname #1\endcsname
\expandafter\let\csname oldend#1\expandafter\endcsname\csname end#1\endcsname
\renewenvironment{#1}
{\linenomath\csname old#1\endcsname}
{\csname oldend#1\endcsname\endlinenomath}}
\newcommand*\patchBothAmsMathEnvironmentsForLineno[1]{
\patchAmsMathEnvironmentForLineno{#1}
\patchAmsMathEnvironmentForLineno{#1*}}
\AtBeginDocument{
\patchBothAmsMathEnvironmentsForLineno{equation}
\patchBothAmsMathEnvironmentsForLineno{align}
\patchBothAmsMathEnvironmentsForLineno{flalign}
\patchBothAmsMathEnvironmentsForLineno{alignat}
\patchBothAmsMathEnvironmentsForLineno{gather}
\patchBothAmsMathEnvironmentsForLineno{multline}
}

\usepackage{scalerel}

\makeatletter
\newcommand{\overrighharpoonup}[1]{\ThisStyle{\vbox {\m@th\ialign{##\crcr
 \rightharpoonupfill \crcr
 \noalign{\kern-\p@\nointerlineskip}
 $\hfil\SavedStyle#1\hfil$\crcr}}}}

\def\rightharpoonupfill{$\SavedStyle\m@th\mkern+0.8mu\cleaders\hbox{$\shortbar\mkern-4mu$}\hfill\rightharpoonuptip\mkern+0.8mu$}

\def\rightharpoonuptip{\raisebox{\z@}[2pt][1pt]{\scalebox{0.55}{$\SavedStyle\rightharpoonup$}}}

\def\shortbar{\smash{\scalebox{0.55}{$\SavedStyle\relbar$}}}
\makeatother

\let\seq=\overrighharpoonup

\usepackage{tikz}
\usepackage{scalerel}

\newsavebox\vdegbox
\savebox\vdegbox{\tikz{
		\draw[black,fill=black] (90:1) circle (.35);
		\draw[black,line width=0.10cm] (210:1) circle (.30);
		\draw[black,line width=0.10cm] (330:1) circle (.30);
		\draw[opacity=0] (0:1.2) circle (0.1);
	}}

\newsavebox\vvbox
\savebox\vvbox{\tikz{
		\draw[black,line width=0.10cm] (90:1) circle (.30);
		\draw[black,fill=black] (210:1) circle (.35);
		\draw[black,fill=black] (330:1) circle (.35);
		\draw[opacity=0] (0:1.2) circle (0.1);
	}}

\newsavebox\pdegbox
\savebox\pdegbox{\tikz{
		\draw[black,line width=0.10cm] (90:1) circle (.30);
		\draw[black,fill=black] (210:1) circle (.35);
		\draw[black,fill=black] (330:1) circle (.35);
		\draw[black,line width=0.28cm ] (210:1) -- (330:1);
		\draw[opacity=0] (0:1.2) circle (0.1);
	}}

\newsavebox\vvvbox
\savebox\vvvbox{\tikz{
		\draw[black,fill=black] (90:1) circle (.35);
		\draw[black,fill=black] (210:1) circle (.35);
		\draw[black,fill=black] (330:1) circle (.35);
		\draw[opacity=0] (0:1.2) circle (0.1);
	}}
\newcommand{\vvv}{\mathord{\scaleobj{1.2}{\scalerel*{\usebox{\vvvbox}}{x}}}}

\newsavebox\evbox
\savebox\evbox{\tikz{
		\draw[black,fill=black] (90:1) circle (.35);
		\draw[black,fill=black] (210:1) circle (.35);
		\draw[black,fill=black] (330:1) circle (.35);
		\draw[black,line width=0.28cm ] (210:1) -- (330:1);
		\draw[opacity=0] (0:1.2) circle (0.1);
	}}
\newcommand{\ev}{\mathord{\scaleobj{1.2}{\scalerel*{\usebox{\evbox}}{x}}}}

\newsavebox\eebox
\savebox\eebox{\tikz{
		\draw[black,fill=black] (90:1) circle (.35);
		\draw[black,fill=black] (210:1) circle (.35);
		\draw[black,fill=black] (330:1) circle (.35);
		\draw[black,line width=0.28cm ] (90:1) -- (330:1);
		\draw[black,line width=0.28cm ] (90:1) -- (210:1);
		\draw[opacity=0] (0:1.2) circle (0.1);
	}}
\newcommand{\ee}{\mathord{\scaleobj{1.2}{\scalerel*{\usebox{\eebox}}{x}}}}
\newcommand{\piee}{\pi_{\ee}}

\newsavebox\eeebox
\savebox\eeebox{\tikz{
		\draw[black,fill=black] (90:1) circle (.35);
		\draw[black,fill=black] (210:1) circle (.35);
		\draw[black,fill=black] (330:1) circle (.35);
		\draw[black,line width=0.28cm ] (90:1) -- (330:1);
		\draw[black,line width=0.28cm ] (90:1) -- (210:1);
		\draw[black,line width=0.28cm ] (210:1) -- (330:1);
		\draw[opacity=0] (0:1.2) circle (0.1);
	}}

\def\red{\rm red}
\def\blue{\rm blue}

\begin{document}

\title[Localised codegree conditions for tight Hamilton cycles]{Localised codegree conditions for tight Hamilton cycles in 3-uniform hypergraphs}

\author[P.\ Ara\'ujo]{Pedro Ara\'ujo}
\address{IMPA, Rio de Janeiro, Brazil}
\email{pedroc@impa.br}

\author[S.~Piga]{Sim\'on Piga}
\address{Fachbereich Mathematik, Universit\"at Hamburg, Hamburg, Germany}
\email{simon.piga@uni-hamburg.de}

\author[M.~Schacht]{Mathias Schacht}
\address{Fachbereich Mathematik, Universit\"at Hamburg, Hamburg, Germany}
\email{schacht@math.uni-hamburg.de}
\thanks{
Large part of this collaboration was carried out while the first author was a visiting Ph.D.\ student
at the University of Hamburg, funded by PDSE program by the Brazilian agency CAPES. The second author was supported by
ANID/CONICYT Acuerdo Bilateral DAAD/62170017 through a Ph.D.\ Scholarship.
The third author was supported by the ERC (PEPCo 724903).}

\begin{abstract}
We study sufficient conditions for the existence of Hamilton cycles
in uniformly dense $3$-uniform hypergraphs. Problems of this type were
first considered by Lenz, Mubayi, and Mycroft for loose Hamilton
cycles and Aigner-Horev and Levy considered it for tight Hamilton
cycles for a fairly strong notion of uniformly dense hypergraphs. We
focus on tight cycles and obtain optimal results for a weaker notion of
uniformly dense hypergraphs. 

We show that if an $n$-vertex $3$-uniform
hypergraph $H=(V,E)$ has the property that for any set of vertices $X$
and for any collection $P$ of pairs of vertices, the number of
hyperedges composed by a pair belonging to $P$ and one vertex from $X$
is at least $(1/4+o(1))|X||P| - o(|V|^3)$ and~$H$ has minimum vertex
degree at least $\Omega(|V|^2)$, then $H$ contains a tight Hamilton
cycle. A probabilistic construction shows that the constant $1/4$ is
optimal in this context.
\end{abstract}

\maketitle

\section{Introduction}

Dirac's theorem states that any graph on~$n\geq3$ vertices and minimum degree at least~$n/2$ contains  a Hamilton cycle. This is best possible in terms of minimum degree, since a graph composed by two disjoint cliques of sizes~$\lfloor n/2 \rfloor$ and~$\lceil n/2\rceil$ is not even connected. Here we investigate what kind of properties ensure the existence of Hamilton cycles in 3-uniform hypergraphs. 

Since we restrict our attention to 3-uniform hypergraphs, if not mentioned otherwise, by a hypergraph we will mean a 3-uniform hypergraph.
We denote an edge~$\{u,v,w\}\in E(H)$ by~$uvw$. 
An ordered set of distinct vertices~$(v_1, v_2, \dots, v_\ell)$ forms a \textit{tight path} of length~$\ell-2$ if 
every three consecutive vertices form an edge. The pairs~$(v_1, v_2)$ and~$(v_{\ell-1}, v_\ell)$ are the \textit{starting pair} and the \textit{ending pair} of the path, and we frequently call such a tight path a~$(v_1,v_2)$-$(v_{\ell-1},v_\ell)$-path. 
For simplicity we  denote a tight path by listing its vertices.
A tight path~$v_1v_2\dots v_\ell$ together with the edges~$v_{\ell-1}v_\ell v_1$ and~$v_\ell v_1v_2$ forms a \textit{tight cycle} of length~$\ell$. 
A tight cycle which covers all vertices of the hypergraph will be called \textit{tight Hamilton cycle}. Similarly, a \textit{loose Hamilton cycle} in an~$n$-vertex hypergraph (with~$n$ even) is a cyclicly ordered collection of~$n/2$ edges in such a way that two edges intersect if and only if they are consecutive and, consequently, they intersect in exactly one vertex. 

There are more than one notion of degrees in hypergraphs. Given a hypergraph~$H$ and~$v \in V(H)$, we define the \textit{neighbourhood} and the \textit{degree} of~$v$ by
$$N_H(v) = \{e \smallsetminus \{v\}\colon v\in e \in E(H)\} \quad \text{ and } \quad d_H(u)=|N(u)|,$$

\noindent respectively. Similarly, for~$u,v\in V(H)$, we also define their \textit{neighbourhood} and their \textit{codegree} by
$$N_H(u,v) = \{w\in V(H) \colon \{u,v,w\}\in E(H)\}\quad \text{ and } \quad d_H(u,v)=|N(u,v)|.$$
Let~$\delta_1(H)$ be the \textit{minimum degree} and~$\delta_2(H)$ the \textit{minimum codegree} of~$H$. 

A possible extension of Dirac's theorem for hypergraphs was proposed in~\cite{KATONA}. The optimal minimum degree and codegree conditions were obtained for \textit{loose} Hamilton cycles~\cites{loosedegree, loosecodegree} and for \textit{tight} Hamilton cycles~\cites{5/9, RRSz}. 
As the extremal examples for Dirac's theorem for graphs, the constructions that show optimality for those results have a very rigid structure. 
In the graph case, for instance, the extremal constructions contain large pairs of sets of vertices with no edges between them.

Motivated by this, we say an~$n$-vertex graph~$G$ is~$(\rho, d)$-dense if for every pair of vertex sets,$~X$ and~$Y$, the number of edges between them is at least~$d|X||Y|-\rho n^2$. Using a result from Chv\'atal and Erd\H{o}s~\cite{chvatalerdos}, it is not hard to prove that for every~$\alpha, d> 0$ there is an~$\rho>0$ for which every sufficiently large~$(\rho, d)$-dense~$n$-vertex graph with minimum degree  at least~$\alpha n$ contains a Hamilton cycle. Note that the minimum degree condition can not be dropped, as this notion of~$(\rho,d)$-density does not prevent the graph from having isolated vertices.

There are several ways to extend the notion of~$(\rho, d)$-density to 3-uniform hypergraphs. Here we consider the following three notions that we symbolise by~$\vvv$,~$\ev$, and~$\ee$ (see also~\cites{cherry, Christian, Someremarks, RRS-JCTB}). 

\begin{dfn}\label{def:vvv}
  Let~$\rho$, $d \in (0,1]$ and let~$H$ be a~$3$-uniform hypergraph on~$n$ vertices.
  
  We say that~$H$ is \emph{$(\rho,d, \vvv)$-dense} if for every three sets of vertices~$X,Y,Z$ we have \[e(X,Y,Z) = |\{(x,y,z)\in X\times Y\times Z\colon  \{x,y,z\}\in E(H)\}| \geq d |X||Y||Z|- \rho n^3.\]
  
  We say that~$H$ is \emph{$(\rho,d,\ev)$-dense} if for every set of vertices~$X$ and every collection of pairs of vertices~$P\subseteq V\times V$ we have \[e(X, P) = |\{(x,(y,z))\in X\times P\colon \{x,y,z\}\in E(H)\}| \geq d |X||P|- \rho n^3.\]
  
  We say that~$H$ is \emph{$(\rho,\varepsilon, \ee)$-dense} if for every two collections of pairs of vertices~$P, Q\subseteq V\times V$ we have \[e(P,Q) = |\{((x,y),(y,z))\in P\times Q \colon \{x,y,z\}\in E(H)\}| \geq d |K_{\ee}(Q,P)|- \rho n^3,\]
  where~$K_{\ee}(Q,P)=\{((x,y),(y,z)) \in P\times Q\}$.
\end{dfn}

Observe that~$\vvv$ is the weakest notion and~$\ee$ is the strongest (see~\cite{RRS-JCTB} for details). Our main result concerns~$\ev$-dense hypergraphs. We consider this notion as a \textit{localised codegree condition} 
since it implies that for every linear sized set~$X$ most pairs of vertices will have the 
same proportion of neighbours in~$X$ as in the whole hypergraph.

We are interested in (asymptotically) optimal assumptions for~$\ev$-dense
hypergraphs to ensure Hamilton cycles. This line of research can be traced back to the work of Lenz, Mubayi and Mycroft~\cite{mubayi}, who proved that for arbitrarily small~$d, \alpha>0$ there is an~$\rho>0$ such that every sufficiently large~$(\rho, d, \vvv)$-dense~$n$-vertex hypergraph with minimum degree~$\alpha n^2$ contains a loose Hamilton cycle (in fact they proved this result for~$r$-uniform hypergraphs for~$r\geq2$). As this density condition is the weakest one, this theorem implies the same result for the stronger notions~$\ev$ and~$\ee$. 

Aigner-Horev and Levy~\cite{cherry} proved the same conclusion for tight cycles, but considering minimum codegree conditions instead of vertex degrees and assuming the strongest density notion~$\ee$. More precisely, they proved that for every~$d, \alpha>0$ there is a~$\rho>0$ such that every sufficiently large~$(\rho,d, \ee)$-dense hypergraph with minimum codegree~$\alpha n$ contains a tight Hamilton cycle. 
It turns out that for the~$\ev$-density an analogous result is not possible due the following counterexample. 

\begin{example}\label{ex:lbb}\rm
Let~$G$ be a random graph~$G_{n-2,1/2}$ and define a $3$-uniform hypergraph on the same set of vertices for which a triple of vertices is a hyperedge, if it forms a triangle in~$G$ or in~$\overline G$. Observe that every tight cycle in~$H$ can only use edges, all of which induce triangles in~$G$ or they induce only triangles en~$\overline G$. 
Finally, add two new vertices~$x,y$ in such a way that~$N_H(x) = E(G)$ and~$N_H(y)=E(\overline G)$. Then~$x$ is covered only by cycles induced by triangles in~$G$ and~$y$ is covered only by cycles induced by triangles in~$\overline G$. Hence~$H$ contains no tight Hamilton cycle. Obviously, adding all the edges containing the pair~$\{x,y\}$, the hypergraph~$H$ only yields a tight Hamilton path, but not a tight Hamilton cycle. One can show  for every that~$\rho>0$ with high probability~$H$ is~$(\rho, 1/4,\ev)$-dense and it has 
minimum degree~$(1/4-\rho)\binom{n}{2}$ and even minimum codegree~$(1/4-\rho)n$. 
\end{example}

Our main result asserts that the previous example is essentially best possible.

\begin{theorem}\label{maintheorem}
  For every~$\eps>0$ there exist~$\rho>0$ and~$n_0$ such that 
  every~$(\rho,1/4+\eps,\ev)$-dense $3$-uniform hypergraph~$H$ on~$n\geq n_0$ vertices 
  with~$\delta_1(H)\geq \eps \binom{n}{2}$ contains a tight Hamilton cycle.
\end{theorem}
We also strengthen a result of Aigner-Horev and Levy~\cite{cherry} by showing that their codegree assumption for tight Hamilton cycles in~$\ee$-dense hypergraphs can be relaxed to a minimum vertex degree assumption.

\begin{theorem}
\label{maintheorem2}
  For every~$d$, $\alpha>0$ there exist~$\rho>0$ and~$n_0$ such that 
  every~$(\rho,d,\ee)$-dense $3$-uniform hypergraph~$H$ on~$n\geq n_0$ vertices 
  with~$\delta_1(H)\geq \alpha \binom{n}{2}$ contains a tight Hamilton cycle.
\end{theorem}

Theorem~\ref{maintheorem2} was conjectured in~\cite{cherry} and was obtained independently 
in~\cite{GH}. 
The main purpose of this paper is proving Theorem~\ref{maintheorem}. The proof of Theorem~\ref{maintheorem2} 
is based on similar ideas and we discuss the details in Section~\ref{SectionCherry}.

The rest of the paper is organised as follows. In Section 2 we recall the Absorption Method and introduce its three main parts the \textit{Almost Covering Lemma}, the \textit{Connecting Lemma} and the \textit{Absorbing Path Lemma}.
The proof of those lemmas are given in Sections~\ref{SectionAC},~\ref{SectionConnecting}, and~\ref{SectionAbsorbingPath}. 
In Section~\ref{SectionPreliminary} we collect some preliminary observations. 
In Section~\ref{SectionCherry} we discuss the necessary changes to the main proof in order to prove Theorem~\ref{maintheorem2}. 
We close with a few concluding remarks in Section~\ref{SectionFurtherRemarks}.

\section{Absorption Method}
\label{SectionOverview}

In~\cite{RRSz}, R\"odl, Ruci\'nski and Szemer\'edi introduced the Absorption Method, which turned out to be a very useful approach for embedding spanning cycles in hypergraphs. This method reduces the problem to finding an almost spanning cycle with a small special path in it, called the \textit{absorbing path}. The absorbing path~$A$ can absorb any small set of vertices into a new bigger path, with the same ends as~$A$, completing the almost spanning cycle into a Hamilton cycle. 

The almost spanning cycle will be composed from smaller tight paths, which will be connected to longer paths. 
For that it would be useful if any given two pairs of vertices~$(x,y)$ and~$(w,z)$, 
being the ends of such smaller paths, can be connected by a short tight path.
However, in view of the assumptions of Theorem~\ref{maintheorem},
it is easy to see that not \textit{any} pair of pairs can be connected in this way (in particular, there could be pairs with codegree zero). 
For that we introduce the following notion of \textit{connectable pairs} and we will show that for those pairs there actually exist tight connecting paths between them (see Lemma~\ref{Connecting_Lemma} below).

\begin{dfn}\label{ConnectablePair}
Let~$H=(V,E)$ be a hypergraph. We say that~$(x,y)\in V\times V$ is \emph{$\beta$-connectable in~$H$} if the set~$$Z_{xy}=\{z\in V\colon xyz \in E(H) \text{ and } d(y,z)\geq \beta |V|\},$$ has size at least~$\beta |V|$.
Moreover, we say that an~$(a,b)$-$(c,d)$-path is~\emph{$\beta$-connectable} if the pairs~$(b,a)$ and~$(c,d)$ are~$\beta$-connectable. 
\end{dfn}
Observe that the starting pair of the path is asked to be~$\beta$-connectable
in the inverse direction that as it appears in the path.

The proof of Theorem~\ref{maintheorem} splits into three lemmas. Let~$H=(V,E)$ be a~$(\rho,1/4+\varepsilon, \ev)$-dense hypergraph on~$n$ vertices, with~$1/n\ll \rho \ll \varepsilon$. First we prove that such hypergraphs can be almost covered by a collection of `few' tight paths. We remark that this is even true under the weaker assumption of non-vanishing~$\vvv$-density. A straight forward proof is presented in Section~\ref{SectionAC}.

\begin{lemma}[Almost Covering Lemma]\label{AC}
For all~$d, \gamma \in (0,1]$ there exist~$\rho, \beta>0$, and~$n_0$ such that in every~$(\rho, d, \vvv)$-dense hypergraph~$H$ on~$n\geq n_0$ vertices
there exists a collection of at most~$1/\beta$ disjoint~$\beta$-connectable paths, that cover all but at most~$\gamma^2 n$ vertices of~$H$.
\end{lemma}

Next we discuss how to find an absorbing path, which contains a collection of several smaller structures, called
\textit{absorbers}. For~$v\in V$, we call~$A_v \subseteq H$ an absorber for~$v$ if both~$A_v$ and~$A_v\cup \{v\}$ span tight
paths with same ends (we say that~$A_v$ absorbs~$v$). 
The main difficulty is to define the absorbers in such a way that we can prove that every vertex is contained in many of them. 
In Section \ref{SectionAbsorbingPath} we see that the absobers considered here are in fact more complicated and absorb sets of three vertices instead of one. This leads to a divisibility issue which we consider separately in Lemma \ref{cyclepath}. 
Going further, we can find a relatively small collection of tight paths which can absorb any sufficiently small given set of vertices. 
After finding this collection we connect them together to form one tight path with the absortion property described in the following lemma. 

\begin{lemma}[Absorbing Path Lemma]\label{AbsorbingPathLemma}
For every~$\eps>0$ there exist~$\rho, \beta, \gamma'>0$ and~$n_0$ such that the following is true for every positive~$\gamma\leq \gamma'$ and every~$(\rho, 1/4 + \eps,\ev)$-dense hypergraph~$H=(V,E)$ on~$n\geq n_0$ vertices with~$\delta_1(H)\geq \eps n^2$. 

For every~$R\subseteq V$, with~$|R|\leq 2\gamma^2n$, there exists a tight~$\beta$-connectable path~$A$ satisfying $V(A)\subseteq V\smallsetminus R$  and~$|V(A)|\leq \gamma n$, such that for every~$U\subseteq V(H)\setminus A$ with~$|U|\leq 3\gamma^2n$, the hypergraph~$H[V(A)\cup U]$ has a tight path with the same ends as~$A$.
\end{lemma}
The set of vertices~$R$ in Lemma~\ref{AbsorbingPathLemma} will act as a reservoir of vertices that will be used later 
for connecting the tight paths mentioned in Lemmas~\ref{AC} and~\ref{AbsorbingPathLemma},
without interfering with the vertices already used by those tight paths.

The next lemma justifies Definition~\ref{ConnectablePair} and shows that between every two~$\beta$-connectable pairs there exist several short tight paths connecting them.
As it was said before, this is used for connecting the absorbers in the proof of Lemma~\ref{AbsorbingPathLemma}.
Moreover, observe that all tight paths mentioned in Lemma~\ref{AC} and~\ref{AbsorbingPathLemma} are~$\beta$-connectable. 
This allows us to connect them together into an almost spanning cycle 
and the absorbing path in this cycle will absorb all the remaining vertices 
to complete the Hamilton cycle.

\begin{lemma}[Connecting Lemma]\label{Connecting_Lemma}
For every~$\eps, \beta>0$ there exist~$\rho,\alpha>0$ and~$n_0$ such that for
every~$(\rho, 1/4+\eps, \ev)$-dense hypergraph~$H$ on~$n\geq n_0$ vertices the following holds. 

For every pair of disjoint ordered~$\beta$-connectable pairs of vertices 
$(x,y), (w,z) \in V\times V$ there exists an integer~$\ell \leq 15$ such that the number of~$(x,y)$-$(z,w)$-paths with~$\ell$ inner vertices is at least~$\alpha n^{\ell}$ 
\end{lemma}
In view of the construction given in Example~\ref{ex:lbb},
one can see that the~$1/4$ in the~$\ev$-density assumption in Lemma~\ref{Connecting_Lemma} cannot be dropped. In that example, there are two classes of pairs that cannot be connected by a tight path (namely the pairs in~$G$ and in~$\overline{G}$), although they are~$\beta$-connectable.
Hence,~$\ev$-density of at least~$1/4$ is required for Lemma \ref{Connecting_Lemma}.

Also Lemma~\ref{AbsorbingPathLemma} requires~$\ev$-density bigger than~$1/4$. In the proof of Lemma~\ref{AbsorbingPathLemma} this assumption will be crucial 
for connecting the so-called absorbers to a tight path, which makes use of Lemma~\ref{Connecting_Lemma}. Moreover, the type of absorbers used here, leads to a `divisibility issue'. It is addressed in Lemma~\ref{cyclepath} for which we also employ the same density assumption.

We now deduce Theorem~\ref{maintheorem} from Lemmata~\ref{AC}\,--\,\ref{Connecting_Lemma}.

\begin{proof}[Proof of Theorem~\ref{maintheorem}]
Given~$\eps>0$ we apply Lemma~\ref{AbsorbingPathLemma} and obtain~$\rho_1$,~$\beta_1$ and~$\gamma'$. 
Lemma~\ref{AC} applied with~$d=1/4$ and~$\gamma = \min\{\gamma', \eps/2\}$ yields~$\rho_2$ and~$\beta_2$. 
Applying Lemma~\ref{Connecting_Lemma} with~$\eps$ and 
$$\beta = \frac{1}{8}\min\{\beta_1, \beta_2\},$$
reveals~$\alpha$ and~$\rho_3$. 
Finally we set 
$$\rho = \min\{\rho_1, \rho_2/8, \rho_3\},$$ and~$n$ be sufficiently large.
Having fixed all constants, let~$H$ be a~$(\rho, 1/4+\eps, \ev)$-dense hypergraph on~$n$ vertices.

We consider a random set~$R\subseteq V$, in which each vertex is present independently with probability~$\gamma^2$. 
For every positive integer~$\ell\leq15$ consider two pairs~$(x,y), (w,z) \in V\times V$ between which there are at least~$\alpha n^{\ell}$ paths with~$\ell$ inner vertices. Let~$Y=Y(\ell, (x,y),(z,w))$ count the number of such paths whose inner vertices are contained in~$R$. We point out that~$Y$ is a function determined by~$n$ independent random variables, each of which can influence the value of~$Y$  by at most~$n^{\ell-1}$. Therefore a standard application of Azuma's inequality (see~\cite{randomgraphs}*{Section~2.4}) implies that
\begin{equation}
	\label{conecR}
		\mathbb{P}\left(Y  \leq \dfrac{\gamma^{2\ell}}{2} \cdot \alpha n^{\ell} \right) = \exp(-\Omega(n))< \dfrac{1}{2}\cdot\dfrac{1}{15n^4},
\end{equation}

\noindent for any fixed~$\ell$,~$(x,y)$, and~$(w,z)$. Moreover, by Markov's inequality we have that
\begin{equation}
	\label{sizeR}
\mathbb{P}\big( |R| \geq 2\gamma^2 n \big) \leq \dfrac{1}{2}.
\end{equation}

Therefore there exists a realisation of~$R$, which from now on will take over the name~$R$, that is not in the event considered in~\eqref{sizeR} and in any of the events considered in~\eqref{conecR} 
(all~$4$-tuples of vertices and values of~$\ell$). 
Since~$\gamma'<\gamma$,~$\rho<\rho_1$, and~$|R|<2\gamma^2n$, Lemma~\ref{AbsorbingPathLemma} ensures that we can find a~$\beta_1$-connectable absorbing path~$A$ of size smaller than~$\gamma n$ and  which does not intersect~$R$.

Let~$V'=V\smallsetminus (V(A)\cup R)$. Since~$|V(A)\cup R|\leq 3\gamma n \leq n/2$, the induced hypergraph~$H[V']$ is~$(8\rho, 1/4+\epsilon,\ev)$-dense. In particular,~$H[V']$ is~$(8\delta,1/4+\eps, \vvv)$-dense and since~$8\rho \leq \rho_2$, Lemma~\ref{AC} implies that there exists a collection of at most~$1/\beta_2$ paths with~$\beta_2$-connectable ends in~$H[V']$ that cover all but at most~$\gamma^2n$ vertices. 

Set~$t=\lfloor1/\beta_2+1\rfloor$ and let~$(P_i)_{i\in [t]}$ be any cyclic ordering of such paths together with the absorbing path. Assume that we were able to find connections in~$R$ between the paths~$P_1, P_2,\dots, P_i$, using inner vertices from~$R$ only. Moreover, we make sure that each connection is made with at most~$15$ inner vertices. Let~$C_i$ be the path that begins with~$P_1$ and ends in~$P_i$ using those connections. Therefore
\[|V(C_i)\cap R| \leq t \cdot 15 = o(n).\]

Now, we want to show that we can connect~$P_i$ with~$P_{i+1}$ to construct~$C_{i+1}$. Observe that all the tight paths from~$(P)_{i\in[t]}$ are~$\beta$-connectable. This follows from the choice~$\beta\leq\beta_1$ for the absorbing path~$A$. From the paths given by Lemma~\ref{AC} we know that they are~$\beta_2$-connectable in~$H[V']$. Owing to~$\beta\leq \beta_2/2$ and~$|V'|\geq n/2$ the~$\beta$-connectibility follows. 

Let~$(x_i,y_i)$ be the ending pair of~$P_i$ and~$(z_i,w_i)$ the starting pair~$P_{i+1}$. Lemma \ref{Connecting_Lemma} implies that, for some~$\ell_i\leq 15$, there exist at least~$\alpha n^{\ell_i}$ tight~$(x_i,y_i)$-$(z_i,w_i)$ paths, each with~$\ell_i$ inner vertices. By the choice of~$R$, the number of~$(x_i,y_i)$-$(z_i,w_i)$ paths of length~$\ell_i+2$ whose inner vertices lie in~$R$ is at least~$\gamma^2\alpha n^{\ell_i}/2$. Since at most~$|V(C_i)\cap R|n^{\ell_i-1}=o(n^{\ell_i})$ such paths contain a vertex from~$C_i$, for sufficiently large~$n$ large enough we can find one tight path disjoint from~$C_i$.

Finally, consider~$C_t$ the final cycle obtained in this process, by connecting~$P_t$ to~$P_1$. Then, as~$C_t$ includes all the tight paths in the almost covering the number of vertices not covered by~$C_t$ is at most
\[|V\smallsetminus V(C_t) | \leq |R| + \gamma^2 n\leq 3 \gamma^2 n.\]

\noindent This finishes the proof, since~$A$ can absorb these vertices into a new path with the same endings.\end{proof}

\section{Preliminary results and basic definitions}
\label{SectionPreliminary}

In this section we collect some preliminary results and introduce the necessary notation.
Given~$\eta, d\in [0,1]$ and a bipartite graph~$G=(V_1\dcup V_2,E)$ we say that~$G$ is \emph{$(\eta, d)$-regular} if for every two sets of vertices~$X\subseteq V_1$ and~$Y\subseteq V_2$ we have
\[
	|e(X, Y)-d|X||Y|| \leq \eta |V_1||V_2|\,.
\]

It is easy to see that every dense graph contains a linear sized bipartite regular subgraph, with almost the same density. That can be proved by a simple application of Szemer\'edi's Regularity Lemma or alternatively by a more direct density increment argument (see~\cite{peng2002holes}). 

\begin{lemma}
\label{regularisation}
For all~$\eta$, $d>0$ there exists some~$\mu>0$ such that for every~$n$-vertex graph~$G$ with~$e(G) \geq  dn^2/2$, there exist disjoint subsets~$V_1$, $V_2\subseteq V(G)$, with~$|V_1|=|V_2|=\lceil\mu n\rceil$ such that the bipartite induced subgraph 
$G[V_1,V_2]$ is~$(\eta, d')$-regular for some $d'\geq d$. \qed
\end{lemma}

For a hypergraph~$H=(V,E)$ recall its \emph{shadow $\partial H$} is the subset of $V^{(2)}$
of those pairs that are contained in some edge of $H$.
For disjoint sets of vertices~$V_1$, $V_2\subseteq V$ with a slight abuse of notation 
we write~$\partial H[V_1, V_2]$ 
for the set of ordered pairs in $V_1\times V_2$ that correspond to unordered pairs in the shadow, i.e., 
\[
	\partial H[V_1, V_2]
	=
	\big\{
		(v_1,v_2)\in V_1\times V_2\colon \{v_1,v_2\}\in \partial H
	\big\}\,.
\]
Given~$\rho,d>0$, a set of ordered pairs of vertices~$P \in V^{2}$, and a subset~$X\subseteq V$ we say that~$H$ 
is~\emph{$(\rho, d, \ev)$-dense over~$(X, P)$} if for every subset of vertices~$X'\subseteq X$ and every subset of pairs~$P'\subseteq P$ we have
\[
	e(X',P')\geq d\,|X'||P'| - \rho\,|X||P|\,,
\]
which is a version of~$\ev$-density restricted to~$P$ and~$X$.
For the next lemma we also need the following concept of restricted vertex neighbourhood. 
Given a vertex~$v\in V$ and a set of ordered pairs~$P\in V^2$ 
we define its \emph{neighbourhood restricted to $P$} by
\[
	N(v,P)= \{(x,y)\in P \colon vxy\in E\}\,.
\]

\begin{lemma}\label{LemmaRegularDegree}
Let~$H=(V,E)$ be a hypergraph, 
$X\subseteq V$ be a set of vertices, and
$P\subseteq V^{2}$. If~$H$ is~$(\rho, d, \ev)$-dense over~$(X,P)$ for some constants $\rho$, $d>0$, then
\[
\big|\big\{x\in X \colon |N(x,P)| <(d-\sqrt\rho)|P|\big\}\big| < \sqrt\rho\,|X|\,.
\]
\end{lemma}
\begin{proof}
Let~$X'\subseteq X$ be the vertices with less than~$(d-\sqrt{\rho} )|P|$ neighbour pairs in~$P$. The definition of~$X'$ and the~$(\rho,d,\ev)$-density of~$H$ over~$(X,P)$
provide the following upper and lower bounds on~$e(X',P)$
\[d|X'||P|-\rho |X||P| \leq e(X',P) \leq (d-\sqrt{\rho} )|P|\cdot |X'|\]
and the desired bound on $|X'|$ follows.
\end{proof}

The following result asserts that hypergraph contains 
subhypergraph with almost the same density and such that every pair of vertices with positive codegree has at least $\Omega(|V|)$ neighbours.
This fact can be proved by removing iteratively the edges which contain a pair with small codegree and 
we omit the details.

\begin{lemma}\label{cleaning}
For every $\beta>0$ and every $n$-vertex hypergraph $H$ there is a 
hypergraph~$H_{\beta}\subseteq H$ on the same vertex set 
with $e(H_\beta)\geq e(H)-\beta n^3$ such that for every pair of vertices~$x$, $y$ either~$d_{H_{\beta}}(x,y)=0$ or $d_{H_{\beta}}(x,y)\geq \beta n$.
In particular, if we have~$d_{H_{\beta}}(x,y)>0$, then $(x,y)$ is $\beta$-connectable in~$H$.\qed
\end{lemma}

Let~$F$ and~$F'$ be two hypergraphs.
We say that~$F$ \textit{contains a homomorphic copy of~$F'$} if there is a function~$\phi$ from~$V(F')$ to~$V(F)$ such that for every edge~$xyz\in E(F')$ we have that~$\phi(x)\phi(y)\phi(z)\in E(F)$. 
We denote this fact as~$F'\xrightarrow{hom} F$ and we recall the following well known consequence from
Erd\H os~\cite{erdos1964extremal}.

\begin{lemma}\label{blowup}
For every~$\xi>0$ and~$k$, $\ell \in \mathbb N$ there is~$\zeta>0$ and~$n_0\in \mathbb N$ such that the following holds. Let~$F$ and~$F'$ be hypergraphs such that~$|V(F)|=k$ and~$|V(F')|=\ell$ and~$F'\xrightarrow{hom} F$. If a hypergraph~$H$ on~$n>n_0$ vertices contains at least~$\xi n^k$ copies of~$F$, then~$H$ contains~$\zeta n^\ell$ copies of~$F'$.\qed
\end{lemma}

We  denote the hypergraph with four vertices and three edges by~$K_4^{(3)-}$. We refer to the vertex of degree three as the \emph{apex}. Glebov, Kr\'a{\soft{l}}, and Volec \cite{K4minus} showed that~$\vvv$-density bigger than~$1/4$ yields the existence 
of a, in fact of many copies of,~$K_4^{(3)-}$.

\begin{theorem}[Glebov, Kr\'a{\soft{l}} \& Volec, 2016]\label{K4-}
For every~$\epsilon>0$ there exist~$\rho$ and $\xi>0$ such that every sufficiently 
large~$(\rho, 1/4+\epsilon, \vvv)$-dense $n$-vertex hypergraph 
contains~$\xi n^4$ copies of~$K_4^{(3)-}$. \qed
\end{theorem}

\section{Almost covering}
\label{SectionAC}

In this section we present a very straightforward proof of Lemma \ref{AC}. 

\begin{proof}[Proof of Lemma~\ref{AC}]
	Given~$d, \gamma>0$ take~$\beta$ and~$\rho$ such that
	$$\beta=\rho=\frac{d\gamma^6}{13}.$$ 
	We show that a maximal collection of~$\beta$-connectable tight paths, each of which having at least~$\beta n$ vertices, must cover all but at most~$\gamma^2 n$ vertices. We do that by showing that in every set~$X\subseteq V(H)$ with at least~$\gamma^2 n$ vertices there exists a~$\beta$-connectable tight path of size~$\beta n$. Indeed, the~$(\rho, d, \vvv)$-density implies that in such a set~$X$, we have
\[e(X) \geq \dfrac{d|X|^3}{6} - \rho n^3,\]

\noindent where we discounted the ordering of triples. In~$H[X]$ we remove, iteratively, every edge that contains an (unordered) pair of vertices with codegree smaller than~$\beta n$. In this way, we remove at most~$\beta n^3$ edges and get a hypergraph with at least
\begin{align*}
    e(X)-\beta n^3  &\geq \frac{d|X|^3}{6} - \rho n^3 - \beta n^3 \\
                    &\geq \left(\frac{d\gamma^6}{6}-\rho - \beta\right)n^3, 
\end{align*}
edges. Owing to the choice of~$\beta$ and~$\gamma$ this hypergraph is not empty. 
Now a tight path with~$\beta n$ vertices can be found in a greedy manner. Moreover, if~$(x,y)$ is a pair contained in such path, then we have that the set
\[Z_{xy}=\{z\in V\colon xyz\in E \text{ and }d(y,z)\geq \beta n\}\] 
has at least~$\beta n$ vertices.
\end{proof}

\section{Connecting Lemma}
\label{SectionConnecting}

We dedicate this section to prove the Connecting Lemma (Lemma~\ref{Connecting_Lemma}). The proof splits into several lemmata. The Connecting Lemma asserts that every ordered connectable pair can be connected to any other ordered connectable pair. In a first step in Lemmata~\ref{025}
and~\ref{cherries} we show that there are many connections between large sets of
unordered pairs (without specifying the order of the ending pairs). 
In fact, these connection can be achieved by paths consisting of only two edges, which we refer to as \emph{cherries} (see Definition~\ref{def:cherry} below). On the price of extending the length by at most two, in Lemma~\ref{connection} we establish that one can even fix the
order of one of the sets of given pairs. On the other hand, this is complemented by Lemma~\ref{lem:turns} showing that there are many pairs of unordered pairs that 
can be connected in any orientation. We call such pairs of pairs
\emph{turnable} (see Definition~\ref{def:turnable} below).

For the proof of the Connecting Lemma we can now start with any given connectable pair~$(x,y)$ 
and move to its second neighbourhood, which is a large
set of ordered pairs. From that set we shall reach many turnable pairs. Similarly,
from any given ending pair $(z,w)$ we also reach many turnable pairs.
These paths give the turnable pairs an orientation, but since the turnable pairs can
be connected in any orientation, we find the desired tight $(x,y)$-$(z,w)$-paths. 
The detailed presentation of this argument renders the proof of the Connecting Lemma, which we defer to the end of this section.

\begin{lemma}\label{025}
For all~$\xi$, $\epsilon \in (0,1]$ 
there exist~$\eta$, $\rho >0$ 
such that the following holds for sufficiently large~$m$. 

Suppose $V_1$, $V_2$, $V_3$ are pairwise disjoint sets of size $m$
and  suppose $G=(V_1 \dcup V_2,P)$ is an $(\eta,\xi)$-regular bipartite graph.
If~$H=(V_1\dcup V_2\dcup V_3,E)$ is a $3$-partite hypergraph 
that is~$(\rho, 1/4+\epsilon, \ev)$-dense over~$(V_3,P)$, then 
\[ 
\big|\partial H[V_1,V_3]\big|+\big|\partial H[V_2,V_3]\big| \geq \left( 1+ \eps \right)m^2\,. 
\] 
\end{lemma}

\begin{proof}
Given $\xi$ and $\eps$ we set
\[
	\rho=\left(\frac{\eps}{21}\right)^2
	\qqand	
	\eta\leq\frac{\xi\eps}{36}\,.
\]
Let $G=(V_1 \dcup V_2,P)$ and $H=(V_1\dcup V_2\dcup V_3,E)$ be given. Since $G$ is bipartite we may view~$P$ as a subset of $V_1\times V_2$ and, hence, as a set of ordered pairs.
Lemma~\ref{LemmaRegularDegree} applied to~$V_3$ and $P$ ensures for the hypergraph $H$
that there are at 
most~$\sqrt{\rho}m$ vertices in $V_3$ with less than~$(1/4+\epsilon-\sqrt{\rho})|P|$ neighbour 
pairs in~$P$. We remove such vertices from~$V_3$ and let~$V'_3$ be the  resulting subset of~$V_3$. 

Consider a fixed vertex~$v_3 \in V'_3$. By the definition of~$V'_3$, we have
\begin{equation}\label{lower}
		|N(v_3,P)|
		\geq 
		\left( \dfrac{1}{4}+\epsilon- \sqrt{\rho}\right)|P|
		\geq 
		\left(\dfrac{1}{4}+\dfrac{15}{16}\epsilon\right)|P|\,.
\end{equation}
For~$i=1$, $2$ we consider the neighbourhood of~$v_3$ in~$\partial H[V_i, V_3]$ defined by
\[
	N_i(v_3)=\big\{v_i\in V_i\colon (v_i,v_3)\in \partial H[V_i, V_3]\big\}
\]
and note that
\[ 
	|N(v_3,P)|
    \leq 
	e_G\big(N_1(v_3),N_2(v_3)\big)\,.
\]
Consequently, the $(\eta, \xi)$-regularity of $G$ yields  
    \begin{align}\label{upper}
        |N(v_3,P)|
        \leq 
        \xi|N_1(v_3)||N_2(v_3)| + \eta m^2\,.
    \end{align}
Combining~\eqref{lower} and~\eqref{upper} with the
upper bound on $|P|$ provided by the regularity of $G$ we obtain
\[
     	4\xi|N_1(v_3)||N_2(v_3)| 
     	\geq 
     	\Big(1 +\frac{15}{4}\epsilon\Big)|P| - 4\eta m^2
		\geq
     	\Big(1 +\frac{15}{4}\epsilon\Big)(\xi-\eta)m^2 - 4\eta m^2
     	\geq  
		\Big(1 +\frac{7}{2}\epsilon\Big)\xi m^2\,,
\]
where the last inequality makes use of the choice of $\eta$. 
Hence, the AM-GM inequality tells us
\[
	\big(|N_1(v_3)| + |N_2(v_3)|\big)^2 
	\geq 
	4\,\big|N_1(v_3)\big|\big|N_2(v_3)\big| 
	\geq \Big(1+\frac{7}{2}\epsilon\Big)m^2
\]
and, consequently, we arrive at
\[
	|N_1(v_3)| + |N_2(v_3)|
	\geq
	\Big(1+\frac{7}{2}\epsilon\Big)^{1/2}m 
	\geq 
	\Big(1+\frac{11}{10}\epsilon\Big)m\,.
\]
Finally, summing for all vertices~$v_3\in V_3'$ we obtain the desired lower bound 
\begin{align*}
	\big|\partial H[V_1,V_3]\big|+\big|\partial H[V_2,V_3]\big| 
	&\geq
	\sum_{v_3\in V_3'}\big(|N_1(v_3)| + |N_2(v_3)|\big)\\
	&\geq 
    \Big(1 + \dfrac{11}{10}\eps\Big)m\cdot |V'_3|\\
	& \geq 
	\Big(1 + \dfrac{11}{10}\eps\Big)\big(1-\sqrt{\rho}\big)m^2\\
	& \geq 
	( 1+ \eps)m^2\,,
\end{align*}
where we used the choice of $\rho$ for last inequality.
\end{proof}

Tight paths of length two will play a special r\^ole in our proof and the following notation will be useful.
\begin{dfn}\label{def:cherry}
Given a hypergraph~$H=(V,E)$ and disjoint sets~$p$, $q\in V^{(2)}$, we say that the edges~$xyz$, $yzw\in E$ form a \emph{$(p,q)$-cherry}, if $p=\{x,y\}$
and $q=\{z,w\}$.

Moreover, given two sets~$P$, $Q\subseteq V^{(2)}$, 
we say that edges $e$, $e'\in E$ form a \emph{$(P, Q)$-cherry}, if they form a $(p,q)$-cherry for some 
disjoint sets $p\in P$ and~$q\in Q$. 
\end{dfn}

The next lemma asserts that in $\ev$-dense hypergraphs with density larger than $1/4$
large sets of pairs induce many cherries.

\begin{lemma}\label{cherries}
For every~$\xi$, $\epsilon \in (0,1]$ 
there exist~$\rho$, $\nu >0$ such that for every sufficiently 
large~$(\rho, 1/4+\varepsilon, \ev)$-dense hypergraph~$H=(V,E)$ the following holds. 
For all sets~$P$, $Q\subseteq V^{(2)}$ 
of size at least~$3\xi n^2$ there are at least~$\nu n^4$ $(P,Q)$-cherries.
\end{lemma}

\begin{proof}
Given~$\xi$ and $\epsilon$ we apply Lemma~\ref{025} and 
we obtain~$\eta$ and~$\rho'$. Without loss of generality we may assume that $\eta\leq \xi/2$.
Moreover, Lemma~\ref{regularisation} applied with~$\eta$ and $d=\xi$
yields some~$\mu>0$ and we fix the desired constants $\rho$ and $\nu$
by
\[
	\rho=\frac{\mu^3\xi}{56}\rho'
	\qqand
	\nu=9\rho^2\mu^4\eps\,.
\]
Let~$H=(V,E)$ and~$P$, $Q\subseteq V^{(2)}$  satisfy the assumptions of the lemma.

We consider a random balanced bipartition of~$A\dcup B$ of~$V$ and let $P_A=\{p\in P\colon p\subseteq A\}$
and $Q_B=\{q\in Q\colon q\subseteq B\}$. A standard application of 
Chebyshev's inequality shows that there exists a balanced partition of $V$ such that $|P_A|$, 
$|Q_B|\geq \xi n^2/2$.
We apply Lemma~\ref{regularisation} separately to the graphs~$(A,P_A)$ and~$(B,Q_B)$ 
and obtain four pairwise disjoint vertex sets~$A_1, A_2\subseteq A$ and 
$B_1, B_2\subseteq B$ each of size~$m\geq\mu n/2$
such that the induced bipartite graphs~$P[A_1, A_2]$ and~$Q[B_1, B_2]$ 
are both~$\eta$-regular with density at least~$\xi$. 

Next for $i=1$, $2$ we consider the $3$-partite subhypergraph $H[B_i,P[A_1,A_2]]$ on 
$A_1\dcup A_2\dcup B_i$ with the edge set
\[
	\big\{\{x,y,z\}\in V^{(3)}\colon x\in B_i\tand \{y,z\}\in E(P[A_1,A_2])\big\}\,.
\]
Lemma~\ref{cleaning} applied to $H[B_i,P[A_1,A_2]]$ with $\beta=\rho$ yields a subhypergraph 
$H^{i,P}_\rho$. Since our choice of $\rho$ guarantees
\[
	\rho n^3+\rho(3m)^3
	\leq
	28\rho n^3
	\leq 
	\rho' \cdot |B_i|\cdot e(P[A_1,A_2])
\]
it follows from the $\ev$-density of $H$, 
that $H^{i,P}_\rho$ is~$(\rho', 1/4+\varepsilon, \ev)$-dense over~$(B_i,P[A_1,A_2])$. 
Similarly, for $i=1$, $2$  we also define the $3$-partite hypergraph~$H^{i,Q}_\rho$ with vertex partition $B_1\dcup B_2\dcup A_i$ and note that it is
$(\rho', 1/4+\varepsilon, \ev)$-dense over~$(A_i,Q[B_1,B_2])$. 

Applying Lemma~\ref{025} to the bipartite graph $P[A_1,A_2]$  and the $3$-partite 
hypergraph $H^{1,P}_\rho$ implies
\[
	\big|\partial H^{1,P}_{\rho}[A_1, B_1]\big|+\big|\partial H^{1,P}_{\rho}[A_{2},B_1]\big|
	\geq
	(1+\eps)m^2.
\]
Moreover, three further applications of Lemma~\ref{025} to 
$P[A_1,A_2]$ with $H^{2,P}_\rho$ and to $Q[B_1,B_2]$ 
with $H^{1,Q}_\rho$ and with $H^{2,Q}_\rho$ show that 
\[
	\sum_{i=1}^2\Big(
	\big|\partial H^{i,P}_{\rho}[A_1, B_i]\big| 
	+
	\big|\partial H^{i,P}_{\rho}[A_2, B_i]\big|\Big) 
	+ 
	\sum_{i=1}^2\Big(
	\big|\partial H^{i,Q}_{\rho}[B_1,A_i]\big|
	+
	\big|\partial H^{i,Q}_{\rho}[B_2,A_i]\big|
	\Big)
	\geq 
	4(1+\varepsilon)m^2.
\]
In particular, rearranging the terms shows that 
\[
	\sum_{i=1}^2\sum_{j=1}^2 \Big(
	\big|\partial H^{j,P}_{\rho}[A_i, B_j]\big| 
	+ 
	\big|\partial H^{i,Q}_{\rho}[B_j,A_i]\big|\Big)
	\geq 
	4(1+\eps)m^2
\]
and, hence, there are some indices $i_0$, $j_0\in \{1,2\}$ such that 
\[
	\big|\partial H^{j_0,P}_{\rho}[A_{i_0}, B_{j_0}]\big| 
	+ 
	\big|\partial H^{i_0,Q}_{\rho}[B_{j_0},A_{i_0}]\big|
	\geq
	(1+\eps)m^2\,.	
\]
Consequently, set of ordered pairs
\[
	R
	=
	\big\{
	\{y,z\}\in V^{(2)}\colon
	(y,z)\in \partial H^{j_0,P}_{\rho}[A_{i_0},B_{j_0}]
	\tand 
	(z,y)\in 
	\partial H^{i_0,Q}_{\rho}[B_{j_0},A_{i_0}]
	\big\}
\]
has size at least $\eps m^2$.

Finally, we note that every $\{y,z\}\in R$ has positive degree 
in both hypergraphs $H^{j_0,P}_{\rho}$ and~$H^{i_0,Q}_{\rho}$ and, hence, these degrees are at least $3\rho m$. Therefore, there are at least $9\rho^2 m^2$ 
distinct vertices $x\in A_{3-i_0}$ and $w\in B_{3-j_0}$ 
such that $xyz$ and $yzw$ form a $(P,Q)$-cherry. 
Summing over all pairs in $R$ yields at least  
\[
	\eps m^2 \cdot 9\rho^2 m^2
	\geq
	\nu n^4
\]
$(P,Q)$-cherries in $H$. 
\end{proof}

The following corollary allows us to find many connections between a large sets of unordered 
and a large set of ordered pairs.

\begin{lemma}\label{connection}
For every~$\xi$, $\epsilon \in (0,1]$ there 
exist~$\zeta$, $\rho >0$ such that for every sufficiently 
large~$(\rho, 1/4+\varepsilon, \ev)$-dense $n$-vertex hypergraph~$H=(V,E)$ the following holds. 

Let~$P\subseteq V\times V$ be a set of ordered pairs
and let~$Q\subseteq V^{(2)}$ be a set of unordered pairs, 
each of size at least~$\xi n^2$. 
There is an~$\ell\in\{2,4\}$ such that there are at 
least~$\zeta n^{\ell+2}$ tight paths of length~$\ell$ which start with an ordered pair from~$P$ 
and ends in (some ordering of) with a pair from~$Q$. 
\end{lemma}

\begin{proof}
Given~$\xi$ and~$\epsilon$ we apply Lemma~\ref{cherries} with $\xi/6$ and $\eps$ 
and obtain~$\rho$ and~$\nu$. Lemma~\ref{blowup} applied for $\nu/2$, $4$, and $6$ 
(in place of $\xi$, $k$, and $\l$ in Lemma~\ref{blowup}) yields the promised constant~$\zeta>0$.
With out loss of generality we may assume that $\zeta<\nu/2$ and let $n$ be sufficiently large.

For a given set of ordered pairs $P\subseteq V\times V$ 
let~$\overline P$ be the set of unordered pairs obtained from~$P$ by ignoring the order. 
In particular, $|\overline{P}|\geq |P|/2\geq \xi n^2/2$ and Lemma~\ref{cherries} asserts 
that there are~$\nu n^4$ different $(\overline P, Q)$-cherries. 
That is to say there are~$\nu n^4$ tight paths on four vertices of the form $xyzw$ 
where~$\{x, y\}\in \overline P$ and~$\{z,w\}\in Q$.

If for~$\zeta n^4$ of those cherries we have that~$(x, y)\in P$, then the lemma follows 
with $\l=2$. Hence, we may assume that for at least~$(\nu-\zeta)n^4\geq \nu n^4/2$ 
of those tight paths we (only) have~$(y, x)\in P$. Consequently, Lemma~\ref{blowup} yields 
$\zeta n^6$ blowups of these two edge paths where the vertices $y$ and $z$ are doubled, i.e., $H$ contains $\zeta n^6$ 6-tuples of distinct vertices $(x,y_1,y_2,z_1,z_2,w)$ such that for every $i$, $j\in\{1,2\}$ we have 
\[
	(y_i,x)\in P\,,\quad 
	\{z_j,w\}\in Q\,,
	\qand 
	xy_iz_jw\ \text{is a tight path with two edges.}
\]
In particular, every such 6-tuple induces a tight path $y_1xz_1y_2wz_2$ which starts with an ordered pair from~$P$ 
and ends in an unordered pair from~$Q$ and this concludes the proof of the lemma.
\end{proof}

For establishing the Connecting Lemma (Lemma~\ref{Connecting_Lemma}) we shall extend
Lemma~\ref{connection} in such a way that we can connect large sets $P$ and $Q$, where both of them consist of ordered pairs. For that certain blowups of~$K_4^{(3)-}$s will be useful 
and we introduce the following notation.

\begin{dfn}\label{def:turnable}
We say a $7$-tuple of distinct vertices $(a_1,a_2, a_3, b_1,b_2, c,d)\in V^7$ is a \emph{turn} in a hypergraph $H=(V,E)$
if for every $i \in \{1, 2,3\}$ and~$j\in \{1,2\}$ the set~$\{a_i, b_j, c, d\}$ 
spans a copy of a~$K_4^{(3)-}$ in $H$ with~$a_i$ being the apex.
\end{dfn}
Combining Theorem~\ref{K4-} and Lemma~\ref{blowup} shows that the hypergraphs with $\vvv$-density bigger than $1/4$ contain many turns. Moreover, we 
observe that in a turn~$T$ the tight paths
\begin{equation}\label{eq:flipu}
	a_1 b_1 c a_2 b_2\,,\quad
	a_1 b_1 c a_3 d b_2 a_2 \,,\quad
	b_1 a_1 c d a_2 b_2\,,\qand
	b_1 a_1 c b_2 a_2
\end{equation}
with at most~$3$ inner vertices connect the pairs~$\{a_1,b_1\}$ and~$\{a_2,b_2\}$ in all four possible orientations. This motivates the following definition.
\begin{dfn}
	For a hypergraph $H=(V,E)$ we say two disjoint unordered pairs $q$, $q'\in V^{(2)}$ 
	are \emph{$(\theta,L)$-turnable}, if for every ordering 
	$(q_1,q_2)$ of $q$ and every ordering $(q'_1,q'_2)$ of $q'$ there exists 
	some positive integer $\l\leq L$ such that the number of tight $(q_1,q_2)$-$(q'_1,q'_2)$-paths 
	in~$H$ with $\l$ inner vertices is at least $\theta |V|^\l$. 
\end{dfn}
It follows from~\eqref{eq:flipu} that pairs $\{a_1,b_1\}$ and $\{a_2,b_2\}$ that are contained in $\Omega(|V|^3)$ turns are $(\theta,3)$-turnable for some sufficiently small $\theta>0$. 
The following variation of this fact, will be useful in the proof of the Connecting Lemma.
\begin{lemma}\label{lem:turns}
	For every $\epsilon \in (0,1]$ there 
	exist~$\theta$, $\rho >0$ such that for every sufficiently 
	large~$(\rho, 1/4+\varepsilon, \vvv)$-dense
	hypergraph~$H=(V,E)$ the following holds. 

	There exists a set $Q\subseteq V^{(2)}$ of size at least $\theta |V|^2$
	such that for every $q\in Q$ there exists a set 
	$Q'(q)\subseteq V^{(2)}$ of size at least $\theta |V|^2$ 
	such that $q$ and $q'$ are $(\theta,3)$-turnable for every $q'\in Q'(q)$.
\end{lemma} 
\begin{proof}
	Let~$H=(V,E)$ be a sufficiently large~$(\rho,1/4+\epsilon,\vvv)$-dense hypergraph 
	on $n$ vertices. A combined application of Theorem~\ref{K4-} and Lemma~\ref{blowup}
	yields a set $\cT\subseteq V^7$ of at least~$\zeta n^7$ turns 
	$(a_1,a_2,a_3,b_1,b_2,c,d)$ in~$H$ for some sufficiently small $\zeta=\zeta(\eps)>0$ and
	we shall deduce the conclusion of the lemma for 
	\[
		\theta=\frac{\zeta}{8}\,.
	\] 
	
	For every pair $(a,b)\in V\times V$ and $i\in\{1,2\}$ let 
	$\cT_i(a,b)$ be the set of such turns where $a$ and $b$ play the r\^oles of $a_i$ and~$b_i$, 	respectively. We consider the set 
	\[
		\cT^{\star}=\big\{(a,a',a_3,b,b',c,d)\in\cT\colon 
		|\cT_1(a,b)\cap\cT_2(a',b')|\geq \zeta n^3/2\big\}
	\]
	and note that $|\cT^{\star}|\geq \zeta n^7/2$.
	By a standard averaging argument there are at least $\zeta n^2/4$ pairs 
	$(a,b)\in V\times V$ for which we have 
	\[
		|\cT_1(a,b)\cap \cT^{\star}|\geq \frac{\zeta}{4} n^5
	\]
	and we denote the set of these ordered pairs by $R$. 
	Note that for every pair $(a,b)\in R$ 
	there is a set $R'(a,b)\subseteq V\times V$ with 
	\begin{equation}\label{eq:Rprime}
		|R'(a,b)|\geq \frac{\zeta}{4} n^2
		\ \text{such that}\
		\big|\cT_1(a,b)\cap\cT_2(a',b')\big|
		\geq 
		\frac{\zeta}{2} n^3
	\end{equation}
	for every $(a',b')\in R'(a,b)$. Finally, let $Q$ be the set of unordered pairs 
	derived from $R$, i.e., 
	\[
		Q=\big\{\{q_1,q_2\}\in V^{(2)}\colon (q_1,q_2)\in R\big\}
	\]
	and for every $q=\{q_1,q_2\}$ set 
	\[
		Q'(q)=\big\{\{q'_1,q'_2\}\in V^{(2)}\colon 
			(q'_1,q'_2)\in R'(q_1,q_2)\cup R'(q_2,q_1)\big\}\,.
	\]
	Clearly,
	\[
		|Q|\geq \frac{|R|}{2}\geq \frac{\zeta}{8} n^2=\theta n^2
		\qand 
		Q'(q)\overset{\eqref{eq:Rprime}}{\geq} \frac{\zeta}{8} n^2=\theta n^2
	\]  
	and the required number of tight paths for every orientation of $q\in Q$ 
	and $q'\in Q'(q)$ follows from~\eqref{eq:flipu} and~\eqref{eq:Rprime}.
\end{proof}
Roughly speaking, the proof of Lemma~\ref{Connecting_Lemma} follows from  
Lemmata~\ref{connection} and~\ref{lem:turns}. The definition of connectable pairs allows us to move from the given ordered pairs $(x,y)$ and~$(w,z)$, that need to be connected, to large sets of ordered pairs $P$, $P'$, by considering their second neighbourhoods.
Moreover, Lemma~\ref{lem:turns} yields sets $Q\subseteq V^{(2)}$ and 
$Q'(q)\subseteq V^{(2)}$ for every $q\in Q$ of turnable pairs. Applying Lemma~\ref{Connecting_Lemma} first 
to $P$ and $Q$ and then to $P'$ and $Q'(q)$ for all $q\in Q$ leads to the desired tight 
$(x,y)$-$(z,w)$-paths. 

\begin{proof}[Proof of Lemma~\ref{Connecting_Lemma}]
For given~$\epsilon$, $\beta>0$ let $\theta$ and $\rho_1$ be the constants provided by Lemma~\ref{lem:turns}.
We set
\[
	\xi=\min\{\theta, \beta^2\}
\]
and Lemma \ref{connection} applied with $\xi$ and $\eps$ yields $\zeta$ and $\rho_2$.
Finally, we define the promised constants 
\[
	\rho=\min\{\rho_1,\rho_2\}\qqand
	\alpha=\frac{\zeta^2\theta}{13}\,.
\]

Let~$H=(V,E)$ be a sufficiently large~$(\rho,1/4+\epsilon,\ev)$-dense hypergraph on $n$ vertices 
and let~$(x,y)$, $(w,z)$ be two disjoint $\beta$-connectable pairs. Consider the second neighbourhoods of these pairs defined by
\begin{equation}\label{eq:CLPP'}
	P=\{(u,v)\in V\times V\colon xyu,\, yuv\in E\}
	\qand
	P'=\{(u',v')\in V\times V\colon wzu',\, zu'v'\in E\}\,.
\end{equation}
Owing to the $\beta$-connectability, both sets $P$ and $P'$ have size at 
least~$\beta^2 n^2\geq \xi n^2$. 

Next, let $Q\subseteq V^{(2)}$ and 
$Q'(q)\subseteq V^{(2)}$ for every $q\in Q$ be the sets of size at least 
$\theta n^2\geq \xi n^2$ provided by Lemma~\ref{lem:turns}. For every $q\in Q$ we denote by 
$P_4(q)$ (resp.\ $P_6(q)$) the number of tight $(u,v)$-$(q_1,q_2)$-paths having $4$ (resp.\ 6)
vertices and $(u,v)\in P$ and $\{q_1,q_2\}=q$. Moreover, we normalise these numbers 
by
\[
	\eta_P(q)=\max\Big\{\frac{P_4(q)}{n^4}\,,\frac{P_6(q)}{n^6}\Big\}
\]
and note that Lemma~\ref{connection} applied to $P$ and $Q$
ensures 
\begin{equation}\label{eq:Pzeta}
	\sum_{q\in Q}\eta_P(q)\geq\zeta\,.
\end{equation}
Analogously, we define $P'_4(q')$, $P'_6(q')$, and $\eta_{P'}(q')$ for every 
$q'\in \bigcup_{q\in Q} Q'(q)$ and Lemma~\ref{connection} applied to $P'$ and $Q'(q)$
implies 
\begin{equation}\label{eq:Pzeta'}
	\sum_{q'\in Q'(q)}\eta_{P'}(q')\geq\zeta\,.
\end{equation}
for every $q\in Q$.
Recall, that the paths accounted for in~\eqref{eq:Pzeta} and~\eqref{eq:Pzeta'}
induce an ordering of the vertices in $q$ and in $q'$. However, by Lemma~\ref{lem:turns} the pairs $q$ and $q'$ are 
$(\theta,3)$-turnable for every $q\in Q$ and $q'\in Q'(q)$, which means that these pairs can be connected for any possible orientation. Consequently, there is some $\l$ with
\[
	5
	\leq
	\l 
	\leq 
	\max\{4,6\}+ \max\{1,2,3\}+ \max\{4,6\}
	=
	15
\]
such that the number of $(x,y)$-$(z,w)$-walks in $H$ is at least 
\[
	\frac{n^\l}{12} \cdot \sum_{q\in Q}\eta_P(q)\cdot \theta\cdot \sum_{q'\in Q'(q)}\eta_{P'}(q')
	\overset{\eqref{eq:Pzeta'}}{\geq}
	\frac{n^\l}{12} \cdot \sum_{q\in Q}\eta_P(q)\cdot\theta\cdot \zeta
	\overset{\eqref{eq:Pzeta}}{\geq}
	\frac{\zeta^2\theta}{12}n^{\l}\,. 
\]
At most $O(n^{\l-1})$ of these walks might not be a path and, hence, the lemma follows for sufficiently large~$n$.
\end{proof}

\section{Absorbing path}
\label{SectionAbsorbingPath}

We dedicate this section to the proof of Lemma~\ref{AbsorbingPathLemma}. Similarly as 
in~\cite{5/9} the absorbers we consider here have two parts. Moreover, we use an idea
of Polcyn and Reiher~\cite{PR}, which reduces the abundant existence of absorbers 
to a degenerate Tur\'an problem 
on the price that we can only absorb exactly three vertices at each time. 

Consider first  the complete~$3$-partite hypergraph~$K^{(3)}_{3,3,3}$ with parts~$A_i=\{x_i,y_i,z_i\}$, for every~$i=1,2,3$. Note that this hypergraph contains the tight paths
\begin{equation}
\label{path}
x_1 x_2 x_3 y_1 y_2 y_3 z_1 z_2 z_3\,,
\end{equation}
and
\begin{equation}
\label{path'}
x_1 x_2 x_3 z_1 z_2 z_3\,.
\end{equation}

This means that from every copy of~$K^{(3)}_{3,3,3}$, ordered as a tight path like in \eqref{path}, we may remove the three inner vertices~$y_1,y_2,y_3$ to obtain a tight path with the same ends.  Since we only consider dense hypergraphs, we can guarantee that many copies~$K^{(3)}_{3,3,3}$ exist. In other words, in such a situation 
the tight path $x_1 x_2 x_3 z_1 z_2 z_3$ could absorb the three vertices $y_1$, $y_2$, and~$y_3$. However, not every triple might be contained in a~$K^{(3)}_{3,3,3}$ and this will be addressed by the second part of the absorbers used here.

Suppose we want to absorb some arbitrary vertices~$v_1$, $v_2$, and~$v_3$. The idea, similarly as in~\cite{5/9}, is to exchange~$v_i$ with~$y_i$ contained in some~$K^{(3)}_{3,3,3}$. Suppose we have found a~$K^{(3)}_{3,3,3}$ as described above, but additionally we find a path (as a graph) on four vertices with edges from~$N_H(v_i)\cap N_H(y_i)$ disjointly for each~$i=1,2,3$. We argue that this whole structure can absorb~$v_1,v_2,v_3$. Indeed, if~$a_ib_ic_id_i$ is a path on four vertices with edges from $N_H(v_i)\cap N_H(y_i)$, then both~$P(v_i) = a_ib_iv_ic_id_i$ and~$P(y_i)= a_ib_iy_ic_id_i$ are tight paths in the hypergraph and with the same endings. Moreover, the minimum degree and the uniform density imply that for each vertex~$v\in V$, most vertices of~$V$ have~$\Omega (n^2)$ common neighbours with~$v$, which is enough to find such paths.

Therefore, if we choose to absorb~$v_1,v_2,v_3$, we will consider the tight paths~$P(v_1)$,~$P(v_2)$, and~$P(v_3)$ and the tight path of~$K^{(3)}_{3,3,3}$ as in \eqref{path}. On the other hand, if we choose not to absorb them, then we consider the tight paths~$P(y_1)$,~$P(y_2)$, and~$P(y_3)$ and the tight path of~$K^{(3)}_{3,3,3}$ as in \eqref{path'}. We will also show that for each triple of vertices, we can find many of these configurations, so that we can choose a small amount of them that still can absorb every triple and also connect them into a single tight path. Observe that this absorbing path can only absorb sets of vertices with size divisible by three, an issue with which we deal later. First we prove that for every triple there are many absorbers.

\begin{dfn}
\label{defabs}
Let~$H=(V,E)$ be a hypergraph and~$(v_1,v_2,v_3)\in V^{3}$. We say
\[
	A=(K,P_1,P_2,P_3)\in V^9\times V^4\times V^4\times V^4\,,
\] 
with~$K=(x_1, x_2, x_3, y_1, y_2, y_3, z_1, z_2, z_3)$ and~$P_i= (a_i,b_i,c_i,d_i)$
is an \emph{absorber for~$(v_1,v_2,v_3)$} if the ordered sets
\begin{enumerate}[label=\rmlabel]
	\item $x_1 x_2 x_3 y_1 y_2 y_3 z_1 z_2 z_3$, $x_1 x_2 x_3 z_1 z_2 z_3$,
	\item $a_ib_iv_ic_id_i$ and~$a_ib_iy_ic_id_i$ for~$i=1,2,3$
\end{enumerate}
induce tight paths in~$H$. 
All hyperedges of those paths that do not include a vertices from~$\{v_1, v_2, v_3\}$
are called \emph{internal edges} of the absorber $A$.
\end{dfn}
Formally absorbers are defined to be four tuples. However, sometimes it will be convenient to view them as $21$-tuples of vertices.
\begin{lemma}
\label{abs}
For all~$d$, $\epsilon \in (0,1]$ there exist~$\rho$, $\xi>0$ such that for 
sufficiently large~$n$ the following holds.

For every ~$(\rho,d, \ev)$-dense hypergraph~$H=(V,E)$ on~$n$ vertices 
with~$\delta_1(H)\geq \epsilon n^2$ and every triple $T=(v_1,v_2,v_3)\in V^{3}$ 
of distinct vertices there are at least~$\xi n^{21}$ absorbers for $T$.
\end{lemma}

\begin{proof}
Given~$d$ and~$\epsilon$ we define some auxiliary constant 
$\zeta=(d/2)^{27}/3$ and set
\[
	\rho=\frac{1}{36}\left(\frac{d}{2}\right)^{54}
	\qqand
	\xi=\frac{\zeta d^9\eps^9}{2^{11}}\,.
\]
Let~$H=(V,E)$ be a~$(\rho, d, \ev)$-dense hypergraph on~$n$ vertices and consider some
triple of vertices $T=(v_1,v_2,v_3)\in V^{3}$. 

Three applications of Lemma~\ref{LemmaRegularDegree} each with $X=V$ and for $i\in[3]$
with the set of ordered pairs 
\[
	\big\{(u,w)\colon \{u,w\}\in N_H(v_i)\big\}
\]
tells us, that there are at most~$3\sqrt{\rho} n$ 
\emph{bad vertices} $v\in V$ that may fail to satisfy 
\begin{equation}\label{eq:notbad}
	\big|N_H(v)\cap N_H(v_i)\big|
	\geq
	(d-\sqrt{\rho})\big|N_H(v_i)\big|
	\geq
	(d-\sqrt{\rho})\delta_1(H)
	\geq
	\frac{d}{2} \eps n^2
\end{equation}
for some $i\in[3]$.
Moreover, the $(\rho, d, \ev)$-density of $H$ implies that the edge density of $H$ is at least $d-2\rho>d/2$ and since the extremal number of any fixed~$3$-partite hypergraph 
is~$o(n^3)$ we have~$K^{(3)}_{3,3,3}\subseteq H$ for sufficiently large~$n$. In fact, 
the standard proof of this fact from~\cite{erdos1964extremal} yields at least 
$((d/2)^{27}-o(1))n^9$ such copies. Consequently, for sufficiently large $n$
there are at least
\[
	\left(\Big(\frac{d}{2}\Big)^{27}-o(1)\right)n^9-3\sqrt{\rho}n\cdot n^8
	\geq
	\zeta n^9
\]
copies of $K^{(3)}_{3,3,3}$ in $H$ that contain no bad vertex. Let $\cK=\cK_T\subseteq V^9$
be the set of these $K^{(3)}_{3,3,3}$ in $H$.

Consider some~$K=(x_1, x_2, x_3, y_1, y_2, y_3, z_1, z_2, z_3)\in \mathcal{K}$.
Since none of the vertices of~$K$ is bad, for every vertex $v$ from $K$ inequality~\eqref{eq:notbad} holds for every $i\in[3]$. In particular, for every $i\in[3]$ we have 
$|N_H(y_i)\cap N_H(v_i)|\geq d\eps n^2/2$ and it follows from~\cite{BR} that there exist at least~$((d\eps/2)^3-o(1))n^4$ 
paths on four vertices with edges from $N_H(y_i)\cap N_H(v_i)$. 
Consequently, for sufficiently large $n$, there exist at least 
\[
	|\cK|\cdot \left(\Big(\frac{d^3\eps^3}{8}-o(1)\Big)n^4\right)^3
	\geq
	\zeta n^9 \cdot
	\frac{d^9\eps^9}{2^{10}}n^{12}
	\geq
	2\xi n^{21}
\]
$4$-tuples $A=(K,P_1,P_2,P_3)\in V^9\times V^4\times V^4\times V^4$ with 
$P_i$ inducing a path in $N_H(y_i)\cap N_H(v_i)$ for~$i=[3]$. Such an $A$ may only 
fail to be an absorber for~$T$, if it contains some vertex from~$T$ itself or if 
its $21$ vertices are not distinct. However, since there are at most $O(n^{20})$ such 
``degenerate'' $A$'s the lemma follows for sufficiently large~$n$.
\end{proof}

Note that for the proof of Lemma~\ref{abs} positive $\ev$-density was sufficient. However,
to address the aforementioned divisibility issue, we will show that the hypergraphs $H$ considered here contain a copy of~$C_8(4)$, the~$4$-blow-up of the tight cycle on~$8$ vertices. 
For the proof of that, we make use of the assumption that the $\ev$-density of $H$ 
is bigger than $1/4$.

The~$C_8(4)$ is formed by~$8$ cyclicly ordered independent sets~$\{e_i, f_i, g_i, h_i\}_{i\in [8]}$ such that the only edges are the ones with vertices from three consecutive such sets. Note that $C_8(4)$ contains the tight path
\begin{equation}
\label{cyclepath}
e_1e_2\dots e_8f_1f_2\dots f_8g_1g_2\dots g_8 h_1h_2\dots h_8.
\end{equation}
Moreover, by removing the sets~$\{f_i\}_{i\in[8]}$ or $\{f_i,g_i\}_{i\in[8]}$ from the path in \eqref{cyclepath} leads to tight paths with the same ends in $C_8(4)$ 
with~$24$ or~$16$ vertices, respectively. 
We also remark that~$16$,~$24$ and~$32$ are congruent to~$1$,~$0$ and~$2$ modulo~$3$, respectively. Therefore, if we connect such tight path to the absorbing path, we can decide to remove some of the vertices so that the size of the leftover set is divisible by~$3$.

\begin{lemma}
\label{blowupcycle}
For all~$\epsilon>0$ there exist~$\rho$, $\theta>0$ such that every sufficiently large
$(\rho, 1/4+\epsilon, \ev)$-dense hypergraph~$H=(V,E)$ contains $\theta |V|^{32}$ copies of~$C_8(4)$.
\end{lemma}

\begin{proof}
Given~$\epsilon>0$ we apply Theorem \ref{K4-} to obtain~$\rho_1$ and~$\xi$. Then, the application of Lemma~\ref{cherries} to~$\xi/6$ and~$\eps$ yields~$\rho_2$ and~$\nu$. Set~$\rho = \min\{\rho_1, \rho_2\}$ and let $n$ be sufficiently large.

Let~$H=(V,E)$ be a~$(\rho,1/4+\epsilon,\ev)$-dense hypergraph on~$n$ vertices. 
In view of Lemma \ref{blowup} it suffices to show that~$H$ contains~$\zeta n^8$ copies of~$C_8$
for some $\zeta>0$.

Theorem~\ref{K4-} implies that~$H$ contains at least~$\xi n^4$ copies of~$K_4^{(3)-}$. Let~$R$ be the set of ordered pairs~$(a,x)$ such that both vertices are contained in at least $\xi n^2/2$
of these $K_4^{(3)-}$ with~$a$ being the apex. By double counting we infer $|R|\geq \xi n^2/2$.

For every~$(a,x)\in R$, let $P_{a,x}\subseteq V^{(2)}$ 
be those pairs $\{y,z\}$ that span such a copy of~$K_4^{(3)-}$ together with $a$ and $x$.
We apply Lemma~\ref{cherries} to~$P=Q=P_{a,x}$ and infer that there are at 
least~$\nu n^4$~$(P,Q)$-cherries, i.e., tight paths with~$4$ vertices starting and ending at a pair from~$P_{a,x}$. 

Let~$F$ be the hypergraph with vertex set~$\{a,x,y,y',z,z'\}$ such that the following holds. 
The vertices~$\{a,x,y,z\}$ and~$\{a,x,y',z'\}$ span copies of~$K_4^{(3)-}$ with apex~$a$ and~$F$ contains a~$(\{y,z\},\{y',z'\})$-cherry. 
Observe that since~$y$ and~$z$ (resp.~$y'$ and~$z'$) play a symmetric role in~$K_4^{(3)-}$, regardless of the orientation of the pairs~$\{y,z\}$ and~$\{y',z'\}$ in the cherry the resulting hypergraph is isomorphic. Without loss of generality we will assume that the cherry is a tight path of the form~$yzy'z'$.  
By the reasoning above,~$H$ contains at least
\[|R|\cdot \nu n^4 \geq \frac{\xi}{2} \nu n^6\]
copies of~$F$. We argue that there is a homomorphism of~$C_8$ in~$F$. 
Indeed, if we consider the vertices of~$F$ in the following cyclic ordering 
$$x a y z y' z' a y'$$
one can check that every consecutive triple forms an edge in~$F$. Since there are at least~$\Omega(n^6)$ copies of~$F$ in~$H$, then by Lemma \ref{blowup} and taking~$\zeta$ small enough, we have that there are at least~$\zeta n^8$ copies of~$C_8$. 
\end{proof}

We are now ready to prove Lemma~\ref{AbsorbingPathLemma}.

\begin{proof}[Proof of Lemma~\ref{AbsorbingPathLemma}]
Given $\eps>0$ the constants appearing in this proof will satisfy the following hierarchy 
\begin{equation}
   \label{constants}
   1>\eps\gg\xi\,,\theta\gg \beta\gg \rho\,,\alpha\gg\gamma'\geq \gamma \gg \frac{1}{n}\,,
\end{equation}
where the auxiliary constants $\xi$, $\theta$, and $\alpha$ are provided by Lemmata~\ref{abs},~\ref{blowupcycle}, and~\ref{Connecting_Lemma} and it is easy to check 
that~\eqref{constants} complies with the quantification of these lemmata.
Let~$H$ be a~$(\rho, 1/4+\epsilon, \ev)$-dense hypergraph with~$\delta_1(H)\geq \epsilon n^2$ 
and let~$R$ be a subset of~$V$ with at most~$2\gamma^2 n$ vertices. Fix the subhypergraph $H_\beta\subseteq H$ provided by Lemma~\ref{cleaning}.

For~$T\in V^{3}$, let~$\mathcal{A}_T$ be the set of those absorbers for~$T$ in $H$ 
that have no vertex in $R$ and all its 36 internal edges from $H_\beta$. It follows from 
Lemma~\ref{abs} applied with $d=1/4+\eps$ and $\eps$ that
\[
	|\mathcal{A}_T|
	\geq
	\xi n^{21} - 21\,|R|n^{20} - 6\cdot 36\big(e(H)-e(H_\beta)\big)n^{18}
	\geq 
	\xi n^{21}- 42\,\gamma^2 n^{21} - 216\,\beta n^{21} 
	\overset{\eqref{constants}}{\geq} 
	\dfrac{\xi}{2} n^{21}\,.
\]
Let $\cA=\bigcup_T\cA_T$ be the union over all triples $T\in V^3$ and 
consider a random collection of absorbers~$\mathcal{C}\subseteq \mathcal{A}$ in which each element of~$\mathcal{A}$ is present independently with probability 
\[
	p=\dfrac{\gamma ^{4/3}n}{2|\mathcal{A}|}\,.
\]
Since~$\EE|\mathcal{A}| = p|\mathcal{A}|$, Markov's inequality ensures that
\begin{equation}
\label{abs1}
\PP\big(|\mathcal{C}|\geq \gamma^{4/3}n\big) \leq \dfrac{1}{2}.
\end{equation}
Moreover, for every~$T\in V^{3}$ we have
\[
	\EE|\mathcal{C}\cap \mathcal{A}_T|=p|\mathcal{A}_T|
	\geq 
	\frac{\gamma ^{4/3}n}{2|\cA|}\cdot\frac{\xi n^{21}}{2}
	\geq
	\dfrac{\gamma^{4/3} \xi n}{4}
	\overset{\eqref{constants}}{\geq} 
	4\gamma^{2} n\,,
\]
Chernoff's inequality combined with the union bound over all triples yields
\begin{equation}
\label{abs2}
	\PP\big( \exists  T\in V^{3}\colon |\mathcal{C}\cap \mathcal{A}_T|< 3 \gamma ^2n \big)
	\leq 
	o(1)\,. 
\end{equation}
Letting $Y$ be the number of pairs of distinct absorbers $A$, $A'\in\cC$ that share a vertex
we note 
\[
	\EE Y 
	= 
	p^2 \cdot n^{21}\cdot 21^2 \cdot n^{20}
	= 
	\frac{\gamma^{8/3} n^2}{4|\mathcal{A}|^2} \cdot 441 n^{41}
	\leq
	\dfrac{441\gamma^{8/3}n}{\xi^2} 
	\overset{\eqref{constants}}{\leq} 
	\dfrac{\gamma^2n}{4}
\]
and by Markov's inequality, we have
\begin{equation}
\label{abs3}
\PP(Y\geq \gamma^2n) \leq \dfrac{1}{4}\,.
\end{equation}

Consequently, with positive probability none of the bad events 
from \eqref{abs1}, \eqref{abs2}, and \eqref{abs3} happen. In particular, there exists a realisation of~$\mathcal{C}$ such that 
\[
	|\cC|< \gamma^{4/3}n\,,\qquad
	|\cC\cap \cA_T|\geq 3 \gamma ^2n\ \text{for every $T\in V^3$,}\qqand
	|Y(\cC)|< \gamma^2n\,.
\]
For every pair of absorbers accounted in $Y(\cC)$ we remove one of the involved absorbers in an arbitrary way and obtain a subset $\cB\subseteq \cC$ of pairwise vertex disjoint absorbers satisfying
\[
|\cB|\leq |\cC|< \gamma^{4/3}n\qqand
|\cB\cap \cA_T| > |\cC\cap \cA_T| - \gamma^2n \geq 2 \gamma ^2n\ \text{for every $T\in V^3$.}
\]
Recall that if the absorbing path would only contain the absorbers from $\cB$, then it could only absorb sets $U$ with a cardinality that is divisible by~$3$. We address this divisibility issue by adding a copy of~$C_8(4)$ to the path. Lemma~\ref{blowupcycle} guarantees at least 
$\theta n^{32}$ copies of~$C_8(4)$ in~$H$. Similarly, as for the estimate of $\cA_T$, 
we infer that there is one such~$C_8(4)$ which is vertex disjoint from the set $R$ and 
from all absorbers from $\cB$ and which only contains edges from~$H_\beta$. In fact, this follows from 
\begin{multline*}
	\theta n^{32}-32\,|R|n^{31}-21\,|\cB|n^{31}-6\cdot e(C_8(4))\big(e(H)-e(H_\beta)\big)n^{29}\\
	\geq 
	\theta n^{32} 
	-
	64\,\gamma^2n^{32} 
	-
	21\,\gamma^{4/3}n^{32}
	-
	3072\,\beta n^{32}
	\overset{\eqref{constants}}{>} 
	0\,.
\end{multline*}
Fix an ordering of the vertices of such a $C_8(4)$ that induces a 
tight path (see, e.g.,~\eqref{cyclepath}) and denote this path by $P_C$.

In order to obtain the final absorbing path, each absorber~$(K,P_1,P_2,P_3)\in \cB$ will be viewed as a collection of four tight paths:~$x_1x_2x_3z_1z_2z_3$ and~$a_ib_iy_ic_id_i$, for~$i=1,2,3$, as in Definition~\ref{defabs}.  
Therefore, together with joining $P_C$ we have to connect~$t=4|\cB|+1$ tight paths to build the promised absorbing path~$A$. 
For each of the connections we will appeal to Lemma~\ref{Connecting_Lemma} and each application will require to add up at most~$15$ inner vertices.

Let~$(P_i)_{i\in[t]}$ be an arbitrary enumeration of all these tight paths that need to be connected.
We continue in an inductive manner starting with $A_1=P_1$, 
let~$A_j$ be the already constructed tight path containing $P_i$ for every $i\leq j$.
Since every connection requires at most $15$ inner vertices and the longest path in 
$(P_i)_{i\in[t]}$ has $32$ vertices we have 
\begin{equation}\label{eq:absAj}
	|V(A_j)|+\sum_{i=j+1}^t|V(P_i)|
	\leq 15(j-1)+32t  
	\leq 47t
	\leq 47\big(4|\cB|+1\big)
	\leq 47\big(4\gamma^{4/3}n+1\big)
	\leq \gamma n\,.
\end{equation}

Suppose now that we want to connect~$P_j$, which ends in~$(x,y)$, to~$P_{j+1}$, which starts at~$(z,w)$. Since all tight paths $P_i$ with $i\in[t]$ have its edges in~$H_{\beta}$, by Lemma~\ref{cleaning}
they are~$\beta$-connectable. Therefore, Lemma~\ref{Connecting_Lemma} implies that there are at least~$\alpha n^{\l}$ tight paths, with~$\l\leq 15$ inner vertices, connecting~$(x,y)$ 
with~$(z,w)$ in~$H$. Consequently, in view of~\eqref{eq:absAj} 
and $|R|\leq2\gamma^2 n$ our choice of $\gamma$ in \eqref{constants} shows that 
at least one of such connecting paths must be vertex disjoint from
\[
	V(A_j)\cup\bigcup_{i=j+1}^t V(P_i)\cup R\,,
\]
which concludes the inductive step and proves the existence of the tight path $A_{j+1}$.

Finally, let~$A=A_t$ be the final tight path and let~$U\subseteq V\smallsetminus V(A)$ 
with~$|U|\leq 3\gamma^2n$. First we remove~$0$,~$8$ or~$16$ vertices from $P_C$ in~$A$ and reallocate them to~$U$ to get a set~$U'$ with size divisible by three. Moreover~$|U'|\leq 3\gamma^2 n + 16\leq 3(\gamma^2n + 6)$ and, hence,~$U'$ can be split into at most~$\gamma^2n + 6$ disjoint triples. Since each triple has at least~$2\gamma^2n>\gamma^2n + 6$ absorbers in~$A$, we can greedily assign one for each and absorb all of them into~$A$. 
\end{proof}

\section{Proof of Theorem~\ref{maintheorem2}}
\label{SectionCherry}

In this section we discuss the few modifications necessary in the proof of Theorem~\ref{maintheorem} in order to prove Theorem~\ref{maintheorem2}. Recall that both theorems have the same minimum vertex degree assumption. However, where Theorem~\ref{maintheorem2} requires the given hypergraph $H$ to be $\ee$-dense for some positive density, Theorem~\ref{maintheorem} requires $\ev$-density bigger than $1/4$. In other words, the uniform density assumptions of both theorems are incomparable.

The proof of Theorem~\ref{maintheorem} consist of three main parts, namely Lemmata~\ref{AC}\,--\,\ref{Connecting_Lemma}. Observe that Lemma~\ref{AC} can be applied directly under the conditions of Theorem~\ref{maintheorem2}, but for Lemmata~\ref{AbsorbingPathLemma} and~\ref{Connecting_Lemma} we have the assumption of~$\ev$-density at least~$1/4$ which is not provided by  Theorem~\ref{maintheorem2}. We start with the discussion of the Connecting Lemma 
in the context of Theorem~\ref{maintheorem2} in the next section and defer the discussion 
of the adjustments for the Absorbing Path Lemma (Lemma~\ref{AbsorbingPathLemma}) to Section~\ref{sec:APL2}. 

\subsection{Connecting Lemma for Theorem~\ref{maintheorem2}}
\label{sec:CLee}
The following lemma will play the r\^ole of Lemma~\ref{Connecting_Lemma} in the proof of Theorem~\ref{maintheorem}.  

\begin{lemma}[Connecting Lemma for~$\ee$-density conditions]\label{Connecting_Lemma2}
For every~$d$, $\beta>0$ there exist~$\rho$, $\alpha>0$ and an~$n_0$ such that for every~$(\rho, d, \ee)$-dense hypergraph~$H$ on~$n\geq n_0$ vertices the following holds. 

For every~$\ell\in \{5,6,7\}$ and for every pair of disjoint ordered~$\beta$-connectable pairs~$(x,y)$, 
$(w,z) \in V\times V$, the number of~$(x,y)$-$(z,w)$-paths with~$\ell$ inner vertices is at least~$\alpha n^{\ell}$. 
\end{lemma}
\begin{proof}[Proof of Lemma~\ref{Connecting_Lemma2} (sketch)]
We begin with the following observation. Let $P$, $P'\subseteq V\times V$ each of 
size at least $\Omega(n^2)$ we show that 
\begin{equation}\label{eq:CLee1}
	\text{there are at least $\Omega(n^5)$
$p$-$p'$-paths with one inner vertex and $p\in P$, $p'\in P'$.}
\end{equation}
Note that 
every~$(\rho, d, \ee)$-dense hypergraph is~$(\rho, d, \ev)$-dense
and in view of Lemma~\ref{LemmaRegularDegree} applied to $P$ and $V$ there is a set 
$X\subseteq V$ such that 
$|X|=\Omega(n)$ and for every $x\in X$ we have~$|N(x,P)|=\Omega(n^2)$.
Similarly, another application of Lemma~\ref{LemmaRegularDegree} to $P'$ and $X$
yields a set $X'\subseteq X$ of size $\Omega(n)$ such that 
\[
	|N(x,P)|=\Omega(n^2)
	\qqand
	|N(x,Q)|=\Omega(n^2)
\]
for every $x\in X'$. Consequently, a standard averaging argument tells us that 
each of the sets 
\[
	Q=\big\{(p_2,x)\in V\times X'\colon 
		|\{p_1\in V\colon (p_1,p_2)\in P\tand p_1p_2x\in E\}|=\Omega(n)\big\}
\]
and 
\[
	Q'=\big\{(x,p'_1)\in X'\times V\colon 
		|\{p'_2\in V\colon (p'_1,p'_2)\in P'\tand xp'_1p'_2\in E\}|=\Omega(n)\big\}
\]
has size $\Omega(n^2)$. Finally, the $\ev$-density of $H$ applied to $Q$ and $Q'$
yields $\Omega(n^5)$ $p$-$p'$-paths starting in $P$ and ending in $P'$ with an
inner vertex from $X$, i.e., it establishes~\eqref{eq:CLee1}.

For given connectable pairs $(x,y)$ and $(w,z)$
letting $P$ and $P'$ be their second neighbourhoods as defined in~\eqref{eq:CLPP'}, yields
the conclusion of Lemma~\ref{Connecting_Lemma2} for $\l=5$. 

For $\l=6$ we note that $\ev$-density implies that there are $\Omega(n^2)$
$\beta'$-connectable pairs $(y,y')$ with $xyy'\in E$ for sufficiently small $\beta'=\beta'(d)>0$.
Applying the same argument as above for every such pair $(y,y')$ proves the case 
$\l=6$. Finally, for $\l=7$ the same reasoning applied to the connectable pairs 
$(y',y'')$ with $xyy'$, $yy'y''\in E$ concludes the proof. 
\end{proof}

\subsection{Absorbing Path Lemma for Theorem~\ref{maintheorem2}}
\label{sec:APL2}
Recall that the proof of Lemma~\ref{AbsorbingPathLemma} required $\ev$-density bigger than 
$1/4$ in only two places:
\begin{enumerate}[label=\rmlabel]
\item\label{it:ee1} for the connection of the absorbers to a tight path and
\item\label{it:ee2} in Lemma~\ref{blowupcycle} for addressing the divisibility issue of the size of the
	absorbable sets,
\end{enumerate}
while for the abundant existence of the absorbers $\ev$-density $d$ for any $d>0$ is sufficient (see Lemma~\ref{abs}).
As shown in Section~\ref{sec:CLee} for the connecting lemma positive $\ee$-density suffices, which addresses~\ref{it:ee1}.
Moreover, in Lemma~\ref{Connecting_Lemma2} we are even free to choose the length of the connecting paths, which renders the divisibility issue from~\ref{it:ee2}
in this context.

\section{Concluding remarks}
\label{SectionFurtherRemarks}
We briefly discuss a few open problems for $3$-uniform hypergraphs 
and possible generalisations of 
Theorems~\ref{maintheorem} and~\ref{maintheorem2} to $k$-uniform hypergraphs.
\subsection{Problems for \texorpdfstring{$3$-uniform}{3-uniform} hypergraphs}
Theorems~\ref{maintheorem} and~\ref{maintheorem2} concern asymptotically 
optimal assumptions for uniformly dense hypergraphs that guarantee
the existence of Hamilton cycles. The following notation will be useful for the further discussion.
\begin{dfn}
Given $\star\in\{\vvv,\ev,\ee\}$ and $a\in\{1,2\}$. We say a pair of reals $(d,\alpha)$ is~{\rm$(\star,a)$-Hamilton} if
the following assertion holds:

For every $\eps>0$ there exist $\rho>0$ and $n_0$ such that every $(\rho,d+\eps,\star)$-dense hypergraph $H=(V,E)$
with $|V|=n\geq n_0$ and $\delta_a(H)\geq (\alpha+\eps)\binom{n}{3-a}$ contains a tight Hamilton cycle.
\end{dfn}
We remark that we can restrict our attention to tight Hamilton cycles, since the result of Lenz, Mubayi, and Mycroft~\cite{mubayi}
asserts that already $(0,0)$ would be $(\star,a)$-Hamilton for loose cycles for every choice of 
$\star\in\{\vvv,\ev,\ee\}$ and $a\in\{1,2\}$.
For tight Hamilton cycles Aigner-Horev and Levy~\cite{cherry} showed that 
$(0,0)$ is $(\ee,a)$-Hamilton for $a=2$ and this was extended by Gan and Han~\cite{GH} and by Theorem~\ref{maintheorem2}  to $a=1$.
It remains to characterise the minimal pairs $(d,\alpha)$ that are $(\star,a)$-Hamilton for the four combinations
$\star\in\{\vvv,\ev\}$ and $a\in\{1,2\}$.

Example~\ref{ex:lbb} shows that for 
$(d,\alpha)$ being $(\ev,1)$-Hamilton we must have
\begin{equation}\label{eq:ev2lb}
\max\{d,\alpha\}\geq \frac{1}{4}\,.
\end{equation}
On the other hand, Theorem~\ref{maintheorem} asserts that for $d=1/4$ already $\alpha=0$ suffices. It would be interesting to determine
the smallest value $\alpha_{\ev,1}$ such that $d=0$ suffices. In view of~\eqref{eq:ev2lb} we have $\alpha_{\ev,1}\geq 1/4$
and the result from~\cite{5/9} bounds $\alpha_{\ev,1}$
by~$5/9$. Since all known lower bound constructions for that result are lacking to be $\ev$-dense
it seems plausible that~$\alpha_{\ev,1}<5/9$.

Similarly, let $\alpha_{\ev,2}$ be the infimum over all $\alpha\geq 0$ such that $(0,\alpha)$ is $(\ev,2)$-Hamilton. Here it follows
from~\cite{RRSz} that $\alpha_{\ev,2}\leq 1/2$. Moreover, 
Example~\ref{ex:lbb} yields a hypergraph with minimum codegree $(1/4-o(1))n$ that fails to contain a tight Hamilton cycle. Therefore, we have $\alpha_{\ev,2}\geq 1/4$ and at this point we are not aware of any reason that excludes the possibility that~$\alpha_{\ev,2}$ matches this lower bound.

\begin{problem}
Determine $\alpha_{\ev,1}$ and $\alpha_{\ev,2}$.
\end{problem}

For tight Hamilton cycles in $\vvv$-dense hypergraphs the problem appears to be more delicate as the following unbalanced version
of Example~\ref{ex:lbb} shows. Instead of a uniformly chosen bipartition of $E(K_{n-2})$ we may colour the edges independently $\red$
with probability $p$ and $\blue$ with probability $1-p$. Let $H_p$ be the resulting hypergraph, where the rest of the construction
is carried out in the same way as in Example~\ref{ex:lbb}. By symmetry we may assume $p\geq 1/2$ and for the same reasons as in Example~\ref{ex:lbb} 
the hypergraph $H_p$ contains no tight Hamilton cycle. Moreover, for every fixed $\rho>0$ we have with high probability that
\[
\delta_1(H_p)=\big(\min\{1-p\,, p^3+(1-p)^3\}-\rho\big)\tbinom{n}{2}
\qand
\delta_2(H_p)=\big((1-p)^2-\rho\big)n
\]
and that $H_p$ is $(\rho,p^3+(1-p)^3,\vvv)$-dense. For $p$ close to 1 this shows that there is no $d<1$ such that
$(d,0)$ is $(\vvv,a)$-Hamilton for $a\in\{1,2\}$. In particular, there is no straightforward analogue of Theorem~\ref{maintheorem} in this setting.

It would be intriguing if this construction is essentially optimal for every~$p\geq 1/2$.
In such an event it would imply a resolution of the following problems, where the lower bound would be obtained from $H_p$ for $p=2/3$ and $p=1/2$.
\begin{problem}
Is it true that:
\begin{enumerate}[label=\rmlabel]
\item $(1/3,1/3)$ is $(\vvv,1)$-Hamilton?
\item $(1/4,1/4)$ is $(\vvv,2)$-Hamilton?
\end{enumerate}
\end{problem}

\subsection{Possible generalisations to \texorpdfstring{$k$-uniform}{k-uniform} hypergraphs}
The notion of tight Hamilton cycles straight forwardly extends to $k$-uniform 
hypergraphs. Moreover, the definition of uniformly dense hypergraphs is inspired from the 
theory of quasirandom hypergraphs (see, e.g.,~\cites{ACHPS,Towsner} and the references therein). 
Below we briefly recall the generalisation of Definition~\ref{def:vvv} for general $k$-uniform 
hypergraphs, where we follow the presentation from~\cite{RRS3}.

Given a nonnegative integer $k$, a finite set $V$, and a set $Q\subseteq [k]$ we write $V^Q$ for the set of all functions from~$Q$ to~$V$. 
It will be convenient to identify the Cartesian power 
$V^k$ with~$V^{[k]}$ by regarding any $k$-tuple $\seq{v}=(v_1, \ldots, v_k)$ as
being the function $i\longmapsto v_i$.
We denote by~$\seq{v}\longmapsto \seq{v}\,|\,Q$ the projection from $V^{k}$ to $V^Q$ 
and the preimage of any set~$G_Q\subseteq V^Q$ is denoted by
\[
	\cK_k(G_Q)=\bigl\{\seq{v}\in V^{k}\colon (\seq{v}\,|\,Q)\in G_Q\bigr\}\,.
\]
We may think of $G_Q\subseteq V^Q$ as a directed hypergraph (where vertices in the directed hyperedges are also allowed to repeat).
More generally, for a subset $\cQ\subseteq\powerset([k])$ of the power set of $[k]$
and a family $\ccG=\{G_Q\colon Q\in \cQ\}$ with $G_Q\subseteq V^Q$ for all $Q\in\cQ$,
we  define
\begin{equation}\label{eq:clique}
	\cK_k(\ccG)=\bigcap_{Q\in\cQ}\cK_k(G_Q)\,.
\end{equation}
Moreover, if $H=(V, E)$ is a $k$-uniform hypergraph on $V$, then $e_H(\ccG)$
denotes the cardinality of the set
\[
	E_H(\ccG)=\bigl\{(v_1, \ldots, v_k)\in \cK_k(\ccG)\colon\{v_1, \ldots, v_k\}\in E\bigr\}\,.
\]
Now we are ready to state the generalisation of Definition~\ref{def:vvv}.

\begin{dfn} \label{dfn:dense}
Let~$\rho$, $d \in (0,1]$, let~$H=(V,E)$ be a~$k$-uniform hypergraph on~$n$ vertices,
and let $\cQ\subseteq\powerset([k])$ be given. We say that $H$ is 
\emph{$(\rho, d, \cQ)$-dense} if for every family $\ccG=\{G_Q\colon Q\in \cQ\}$ associating with each $Q\in\cQ$ 
some $G_Q\subseteq V^Q$ we have
\[
	e_H(\ccG)\ge d\,|\cK_k(\ccG)|-\rho n^k\,.
\]
\end{dfn}  
It is easy to check that for $k=3$ the following subsets of $\powerset([3])$ 
\[
	\cQ_{\vvv}=\big\{\{1\},\{2\},\{3\}\big\}\,,\quad
	\cQ_{\ev}=\big\{\{1\}, \{2, 3\}\big\}\,,\qand
	\cQ_{\ee}=\big\{\{1, 2\}, \{1, 3\}\big\}
\]
correspond to $\vvv$-, $\ev$-, and $\ee$-dense hypergraphs. More precisely, for 
every $\star\in\{\vvv,\ev,\ee\}$  we have that 
a $3$-uniform hypergraph is $(\rho,d,\star)$-dense if and only if it is 
$(\rho,d,\cQ_\star)$-dense.

Example~\ref{ex:lbb} straight forwardly extends to $k$-uniform hypergraphs. 
In fact, we may consider a random bipartition $G\dcup\overline{G}$ of the $(k-1)$-element subsets 
of an $(n-2)$-element set 
and we define a $k$-uniform hypergraph containing only those hyperedges 
with the property that all of its $(k-1)$-element subsets are in the same partition class. 
Finally, we may add two vertices~$x$ and $y$ such that the $(k-1)$-uniform link 
of~$x$ is $G$ and the $(k-1)$-uniform link of $y$ is $\overline{G}$. We remark 
that for $k=2$ this construction leads to two disjoint cliques with $\sim n/2$ vertices, which is a lower bound construction for Dirac's theorem~\cite{Dirac} in graphs. 

It is easy to check that 
the resulting $k$-uniform hypergraph$~H$ does not contain a tight Hamilton cycle and for every fixed $\rho>0$ it is $(\rho, 2^{1-k},\cQ)$-dense for
\[
	\cQ=\big\{Q\in [k]^{(k-2)}\colon 1\in Q\big\}\cup \big\{\{2,\dots,k\}\big\}
\] 
with high probability for sufficiently large $n$. Note that for $k=3$ we have $\cQ=\cQ_{\ev}$ 
and~$H$ provides a lower bound for Theorem~\ref{maintheorem}. It seems plausible that the 
hypergraph~$H$ is essentially optimal for $\cQ$-dense hypergraphs also for $k>3$, i.e., that 
$\cQ$-dense $k$-uniform $n$-vertex hypergraphs with density bigger than $2^{1-k}$
and minimum vertex degree $\Omega(n^{k-1})$ contain a tight Hamilton cycle. This would 
be an interesting extension of Theorem~\ref{maintheorem} to $k$-uniform hypergraphs.

Moreover, one can check that for
\[
	\cQ'=\big\{\{1,\dots,k-1\},\{1,\dots,k-2,k\}\big\}
\]
the hypergraph $H$ constructed above is not 
$(\rho, d,\cQ')$-dense for any fixed $d>0$ and sufficiently small $\rho>0$. Note that for 
$k=3$ we have $\cQ'=\cQ_{\ee}$ and, in fact, Theorem~\ref{maintheorem2} asserts that 
$(\rho,d,\cQ')$-dense hypergraphs with minimum vertex degree $\Omega(n^{2})$
contain a Hamilton cycle for any $d>0$ and sufficiently small $\rho$. 
We remark that the proof of Theorem~\ref{maintheorem2} discussed in 
Section~\ref{SectionCherry} extends to $k$-uniform $\cQ'$-dense hypergraphs with an appropriate 
minimum vertex degree condition.

\end{document}